\documentclass[a4paper,11pt,twoside,reqno]{amsart}

\usepackage{amsfonts,amsthm,amsmath,latexsym,amssymb, esint, graphicx, upref}
\usepackage[english]{babel}
\usepackage[utf8]{inputenc}
\usepackage[svgnames]{xcolor}
\usepackage[T1]{fontenc}

\definecolor{azul}{rgb}{0.1,0.6,0.86}
\usepackage[colorlinks=true,linkcolor=azul,citecolor=azul,urlcolor=azul]{hyperref}

\theoremstyle{plain}
\newtheorem{teo}{Theorem}[section]
\newtheorem{lem}[teo]{Lemma}
\newtheorem{cor}[teo]{Corollary}
\newtheorem{pro}[teo]{Proposition}
\newtheorem{claim}{Claim}
\newtheorem{teoremaintro}{Theorem}

\newtheorem{charintro}{Corollary}

\theoremstyle{definition}
\newtheorem{fed}{Definition}[section]

\theoremstyle{remark}
\newtheorem{rem}[teo]{Remark}

\DeclareFontFamily{U}{BOONDOX-cal}{\skewchar\font=45 }
\DeclareFontShape{U}{BOONDOX-cal}{m}{n}{
  <-> s*[1.05] BOONDOX-r-calo}{}
\DeclareFontShape{U}{BOONDOX-cal}{b}{n}{
  <-> s*[1.05] BOONDOX-b-calo}{}
\DeclareMathAlphabet{\mathcalboondox}{U}{BOONDOX-cal}{m}{n}
\SetMathAlphabet{\mathcalboondox}{bold}{U}{BOONDOX-cal}{b}{n}
\DeclareMathAlphabet{\mathbcalboondox}{U}{BOONDOX-cal}{b}{n}

\newcommand{\Z}{\mathbb{Z}}
\newcommand{\R}{\mathbb{R}}

\newcommand{\E}{\mathbb{E}}
\renewcommand{\H}{\mathcal{H}}

\DeclareMathOperator{\Per}{Per}
\DeclareMathOperator{\Var}{Var}
\DeclareMathOperator{\Cov}{Cov}
\DeclareMathOperator{\sinc}{sinc}
\DeclareMathOperator{\divergence}{div}

\begin{document}

\title{Functions of Bounded Variation and Point Processes}

\author{J.~Antezana}
\address{Departament de Matemàtiques i Informàtica, Universitat de Barcelona, Barcelona, Spain and
Centre de Recerca Matemàtica, Barcelona, Spain}
\email{\href{mailto:jorge.antezana@ub.edu}{\texttt{jorge.antezana@ub.edu}}}

\author{M.~Levi}
\address{Departament de Matemàtiques i Informàtica, Universitat de Barcelona, Barcelona, Spain.}
\email{\href{mailto:matteo.levi@ub.edu}{\texttt{matteo.levi@ub.edu}}}

\author{J.~Marzo}
\address{Departament de Matemàtiques i Informàtica, Universitat de Barcelona, Barcelona, Spain and
Centre de Recerca Matemàtica, Barcelona, Spain}
\email{\href{mailto:jmarzo@ub.edu}{\texttt{jmarzo@ub.edu}}}

\author{J.~Ortega-Cerdà}
\address{Departament de Matemàtiques i Informàtica, Universitat de Barcelona, Barcelona, Spain and
Centre de Recerca Matemàtica, Barcelona, Spain}
\email{\href{mailto:jortega@ub.edu}{\texttt{jortega@ub.edu}}}

%\date{\today}

\thanks{The authors are supported by the grant PID2024-160033NB-I00 funded by MICIU/AEI/10.13039/501100011033 and by FEDER, UE
\\
JOC is also supported by the 2024 ICREA 00142 grant by the Generalitat de
Catalunya}

\begin{abstract}
We investigate the relationship between the analytical properties of functions of bounded variation  and the statistical behavior of hyperuniform point processes. We establish several formulas for the jump part of the gradient of a bounded variation function, extending and unifying previous results by Beretti--Gennaioli and Dávila. In particular, we provide new expressions for the $L^2$-jump of the gradient using both difference quotients and Fourier transform methods. 

Furthermore, we connect these analytic structures to the theory of hyperuniform point processes. By analyzing the variance of linear statistics associated with bounded variation functions, we provide asymptotic estimates that depend on the specific classification of the hyperuniformity of the point process. The results show how the regularity and jump discontinuities of a function dictate the growth rate of fluctuations in point processes.

Finally, we introduce an averaged quadratic BMO-type oscillation functional over translated and rotated cube partitions, similar to one recently studied by Ambrosio et al., and prove, using results from point processes, that it converges to an explicit dimensional constant times the $L^2-$jump, giving in particular a further new characterization of the perimeter of a set.
\end{abstract}

\maketitle

\section{Introduction and main results}

Let $\mathrm{BV}(\R^d)$ denote the space of functions of bounded variation, i.e., the space of functions $u\in L^1(\R^d)$ whose distributional gradient $Du$ is a finite vector-valued Radon measure, or equivalently, for which the following semi-norm is finite,
\[
\|u\|_{\mathrm{BV}}=\sup \left\{\int_{\R^d} u(x)\,\divergence  \varphi(x) \, dx: \,\varphi\in C^1_c(\R^d;\R^d),\quad \|\varphi\|_{\infty}\leq 1\right\}.
\]

This paper investigates several asymptotic characterizations of the total variation and of the jump part of the derivative of a function of bounded variation. We first relate the Bourgain--Brezis--Mironescu--Dávila nonlocal formula for the $BV$ seminorm to a recent Fourier characterization of finite-perimeter sets due to Beretti and Gennaioli. This yields a streamlined proof of their criterion. We then prove a family of Fourier and nonlocal difference-quotient formulas which recover the $L^2$-size of the jump part of $Du$. These analytic results are applied to the asymptotic variance of linear statistics of hyperuniform point processes. Conversely, using ideas from point-process theory, we obtain a new averaged BMO-type characterization of the jump part of $Du$.

We now introduce the main objects and state our results.

\subsection{BV seminorms, jumps and perimeters}
The space $W^{1,1}(\R^d)$, which consists of functions whose distributional gradient is absolutely continuous with respect to the Lebesgue measure and integrable, is a subset of $\mathrm{BV}(\R^d)$, and the inclusion is strict since, for instance, $W^{1,1}(\R^d)$ does not contain any characteristic function of any set of positive measure, while $\mathrm{BV}(\R^d)$ does.

A Lebesgue measurable set $\Omega\subset \R^d$ for which $\chi_\Omega\in \mathrm{BV}(\R^d)$ is called a (global) Caccioppoli set and,
in this case, $\|\chi_\Omega\|_{\mathrm{BV}}$ is called the (global) perimeter of $\Omega$, which we will also denote by $\Per(\Omega)$, a terminology justified by the structure theorem of De Giorgi (see Appendix \ref{Appendix}). Observe that, according to our definition, a Caccioppoli set has always finite measure.

The $\mathrm{BV}$ semi-norm also admits the following expression, which does not explicitly involve distributional derivatives.

\begin{teoremaintro}[\cite{Davila, Brezis}]\label{teo Davila}
    Let $\rho:\R^d\to\R$ be an integrable radial function\footnote{In the original theorem in \cite{Davila, Brezis}, $\rho$ is assumed to be positive, but one can split $\rho = \rho^+ - \rho^-$ and apply the original result to $\rho^+/\|\rho^+ \|_1$ and $\rho^-/\|\rho^- \|_1$ to extend their result to functions $\rho$ that are not necessarily positive.}
with integral one, let $\rho_L(x)=L^d\rho(L x)$. If $u\in BV(\R^d)$,
\begin{equation}\label{eq Davila}
 \lim_{L \to \infty} \iint_{\R^d\times\R^d}
\frac{\bigl|u(x)-u(y)\bigr|}{|x-y|}\,\rho_L(x-y)\,dx\,dy =K_{d}\|u\|_{\mathrm{BV}},
\end{equation} 
where 
\[
K_{d} = \fint_{\mathbb S^{d-1}} |\langle e, \omega\rangle| \,
d\sigma_{d-1}(\omega) =\frac{ \Gamma(d/2)}{\sqrt{\pi}\Gamma\left(\frac{d+1}{2}\right)}, \] 
and $d\sigma_{d-1}$ is the  surface measure on the unit sphere.
\end{teoremaintro}

This result appeared as an open question in a renowned paper by Bourgain, Brezis and Mironescu \cite{BBM}, and was later proved by Dávila \cite{Davila}.\footnote{In \cite{Davila}, the author proved the result for a more general family of mollifiers, but in this work we will keep our attention on this special type of family, which are dilations of one single function.} Strictly speaking, Dávila proved the result for $\mathrm{BV}(\Omega)$, with $\Omega$ a bounded open subset of $\R^d$ with Lipschitz boundary. For a direct proof of the version stated here, where $\Omega = \mathbb{R}^d$, see, for instance, \cite[Theorem 3]{Brezis}.

Observe that Theorem~\ref{teo Davila} in particular provides the following formula for the perimeter of Caccioppoli sets of finite measure: 

%if $\Omega\subset \R^d$ has finite measure and \eqref{liminf Davila} holds for $u=\chi_\Omega$, then $\Omega$ is a Caccioppoli set, and conversely, if $\Omega$ is a Caccioppoli set then \eqref{eq Davila} holds for $u=\chi_\Omega$, hence providing a formula for $\Per(\Omega)=\Vert \chi_\Omega\Vert_{BV}$.

\begin{charintro}\label{char Davila}
Let $\rho_L$ as in Theorem~\ref{teo Davila}. If $\Omega$ is a Caccioppoli set,
\begin{equation}\label{eq Davila caccio limit}
 \lim_{L \to \infty} \iint_{\R^d\times\R^d}
\frac{\bigl|\chi_{\Omega}(x)-\chi_{\Omega}(y)\bigr|}{|x-y|}\,\rho_L(x-y)\,dx\,dy=K_{d}\Per(\Omega).
\end{equation} 
\end{charintro}

Recently, Beretti and Gennaioli \cite{BerGen2024} gave an alternative formula for the perimeter of Caccioppoli sets of finite measure, in terms of Fourier transforms of characteristic functions of such sets.
\begin{charintro}[{\cite[Corollary 1.4]{BerGen2024}}]\label{char Beretti}
If $\Omega\subset \R^d$ is a Caccioppoli set,
\begin{equation}\label{eq perimetro con fourier 1}
\lim_{R\to\infty} \frac {2\pi^2}{R} \int_{B(0,R)} |\xi|^2 |\widehat{\chi_\Omega}(\xi)|^2\, d\xi=\Per(\Omega).
\end{equation} 
\end{charintro}
We point out that the convention we use for the Fourier transform (in the above theorem and in the rest of the paper) is the following: given $f$ in the Schwartz class $\mathcal{S}(\R^d)$, we write
\[
\widehat f(\xi)=\int_{\R^d} e^{-2\pi i x\cdot\xi}\,f(x)\,dx \quad\text{and}\quad f(x)=\int_{\R^d} e^{2\pi i x\cdot\xi}\,\widehat f(\xi)\,d\xi.
\]
With this convention, Plancherel's identity reads as
\[
\int_{\R^d} |f(x)|^2\,dx = \int_{\R^d} |\widehat f(\xi)|^2\,d\xi.
\]
Our first contribution is to show that Corollary~\ref{char Beretti}, which the authors of \cite{BerGen2024} obtained with an independent method, can be in fact deduced from the well known Corollary~\ref{char Davila}. Hence, our first result is the following.
\begin{claim}\label{From Davila to BerGenn}
    The formula \eqref{eq perimetro con fourier 1} appearing in Corollary \ref{char Beretti} can be deduced from \eqref{eq Davila caccio limit} appearing in Corollary \ref{char Davila}.
\end{claim}

We point out that formula \eqref{eq perimetro con fourier 1} in Corollary~\ref{char Beretti} was deduced in \cite{BerGen2024} from the following more general result.

\begin{teoremaintro}[{\cite[Theorem 1.2]{BerGen2024}}]\label{teo Beretti}
Let $u\in \mathrm{BV}(\R^d)\cap  L^\infty(\R^d)$. Then,
\begin{equation}\label{eq BeGe}
\lim_{R\to\infty} \frac {2\pi^2}{R} \int_{B(0,R)} |x|^2 |\widehat{u}(x)|^2\, dx=\mathcal{J}(u).
\end{equation} 
\end{teoremaintro}
In order to explain what $\mathcal{J}(u)$ stands for in the above theorem, we shall recall that for $u\in BV(\R^d)$, according to the Lebesgue decomposition theorem, the measure $Du$ admits a unique decomposition into two mutually singular parts: $Du = D^au + D^su$, where
$D^au=\nabla u \, \mathcal{L}^d$ is absolutely continuous with respect to the Lebesgue measure $\mathcal{L}^d$, and $D^s u$ is the singular part. The latter further splits into two mutually singular measures, finally giving the gradient decomposition (see \cite[Sec.3.7]{AFP00})
\begin{equation}\label{Decomposition of gradient - compact}
    Du=\nabla u \, \mathcal{L}^d+D^ju+D^cu.
\end{equation}
The term $D^j u$ is usually called the jump part of the gradient, and it is, informally speaking, supported on the points where the function jumps, hence the name, across a $(d-1)$-dimensional surface. The term $\mathcal{J}(u)$ appearing in Theorem~\ref{teo Beretti} denotes the $L^2$-jump of the function $u$, that is, loosely speaking, the $L^2$ norm of $D^ju$. In particular, if $\Omega$ is a Caccioppoli set of finite measure, $\mathcal{J}(\chi_\Omega)=\Per(\Omega)$, which clarifies how \eqref{eq perimetro con fourier 1} is a special case of \eqref{eq BeGe}.
%In particular, if $u$ is the characteristic function of a smooth set $E$ (which in particular is a Caccioppoli set), $D^j u$ is the Hausdorff $d-1$-dimensional measure, $\mathcal H^{d-1}$, restricted to the boundary of $E$.  
We refer the reader to the Appendix \ref{Appendix} (and to the references cited there) for a more precise definition of $\mathcal{J}(u)$ and for a more rigorous discussion about the decomposition \eqref{Decomposition of gradient - compact}, which includes also a brief explanation of $D^cu$, the so called Cantor part of the gradient, which does not play any direct role in this paper.

It is worth noting the possibly surprising fact that Dávila's asymptotic formula \eqref{eq Davila} and the $L^2$ asymptotic formula \eqref{eq BeGe}, although equivalent when applied to characteristic functions of Caccioppoli sets, lead to two different quantities when applied to arbitrary functions of bounded variation: to the total variation of the whole gradient, the first one, and to the $L^2$-norm of the jump part of the gradient, the second one.

A natural question is then whether or not there exist formulas for the norm of the jump in terms of difference quotients, variations of Dávila's formula \eqref{eq Davila}. This result
%and also the application to point processes
is obtained as a consequence of the following generalization of Beretti--Gennaioli formula \eqref{eq BeGe}.

\begin{teo}\label{Jump con Fourier}
Let $\mathcalboondox{s}:[0,\infty)\to\R$ be a bounded continuous function which, for some $\alpha\ge 1$, satisfies
\[
\mathcalboondox{s}(t)=t^\alpha + o(t^\alpha) \quad{\text{as } t\to 0}.
\]

If $u\in \mathrm{BV}(\R^d)\cap  L^\infty(\R^d),$ then: 
\begin{enumerate}
\item[i.)] If $\alpha>1$
\begin{equation}\label{eq Fourier Jump1}
\lim_{R\to\infty} R  \int_{\R^d} |\widehat{u}(\omega)|^2 \mathcalboondox{s}(|\omega|/R)\, d\omega =   \left( \int_0^\infty \frac{\mathcalboondox{s}(t)}{t^2} \, dt\right) \, \frac{\mathcal{J}(u)}{2\pi^2}\,.
\end{equation}

\item[ii.)] If  $\alpha=1$
\begin{equation}\label{eq Fourier Jump2}
\lim_{R\to\infty}\frac{R}{\log R} \int_{\R^d} |\widehat{u}(\omega)|^2 \mathcalboondox{s}(|\omega|/R)\, d\omega = \frac{\mathcal{J}(u)}{2\pi^2}  .
\end{equation}
\end{enumerate}
\end{teo}

Theorem~\ref{Jump con Fourier} implies the following difference-quotient formulas, which can be considered variations of Dávila's formula \eqref{eq Davila}.

\begin{teo}\label{Jump con diferencias}
Let $\rho:\mathbb{R}^d \to \mathbb{R}$ be an integrable radial function, and let $g:[0,\infty)\to\mathbb{C}$ be defined by
\[
\widehat{\rho}(\omega)=g(|\omega|).
\]
Assume that for some $\alpha\ge1$,
\begin{equation}\label{cerca del cero}
    g(t)=g(0)-t^\alpha + o(t^\alpha)\qquad \text{as } t\to0^+.
\end{equation}
As usual, for $L>0$ let $\rho_L$ be the standard dilation of $\rho$, that is  
\[
\rho_L(x)=L^d\,\rho(Lx).
\]
Let $u\in BV(\mathbb{R}^d)\cap L^\infty(\mathbb{R}^d)$. Then the following asymptotics hold:

\begin{enumerate}
\item[i)] If $\alpha>1$, then
\[
\lim_{L\to\infty} L\,\iint_{\mathbb{R}^d\times\mathbb{R}^d}
(u(x)-u(y))^2\,\rho_L(x-y)\,dx\,dy
= \frac{1}{\pi^2}
C_\rho\mathcal{J}(u),
\]
where $C_\rho$ is the finite\footnote{The convergence of the integral \eqref{integral finita de s} is guaranteed by \eqref{cerca del cero} near the origin, and at infinity by the fact that $\rho \in L^1$, hence $g$ is bounded.} constant given by
\begin{equation}\label{integral finita de s}
C_\rho:= \int_0^\infty \frac{g(0)-g(t)}{t^2} \, dt.
\end{equation}
\item[ii)] If $\alpha=1$, then
\[
\lim_{L\to\infty}\frac{L}{\log L}\,\iint_{\mathbb{R}^d\times\mathbb{R}^d}
(u(x)-u(y))^2\,\rho_L(x-y)\,dx\,dy
= \frac{\mathcal{J}(u)}{\pi^2}.
\]
\end{enumerate}
\end{teo}

An immediate Corollary is the following.

\begin{cor}
    Let $\rho$ as in the statement of Theorem~\ref{Jump con diferencias} and $\eta(x)=|x|^2\rho(x)$. Let $u\in BV(\mathbb{R}^d)\cap L^\infty(\mathbb{R}^d)$. Then the following asymptotics hold:
    \begin{enumerate}
\item[i)] If $\alpha>1$, then
\begin{equation}\label{eq DD Jump1}
\lim_{L\to\infty} \iint_{\R^d\times\R^d}
\frac{\bigl|u(x)-u(y)\bigr|^2}{|x-y|^2}\,L^{d-1}\eta(L(x-y))\,dx\,dy=C_\rho\frac{\mathcal{J}(u)}{\pi^2}.
\end{equation}
\item[ii)] If  $\alpha=1$ then 
\begin{equation}\label{eq DD Jump2}
\lim_{L\to\infty}\frac{1}{\log L} \iint_{\R^d\times\R^d}
\frac{\bigl|u(x)-u(y)\bigr|^2}{|x-y|^2}\,L^{d-1}\eta(L(x-y))\,dx\,dy=\frac{\mathcal{J}(u)}{\pi^2} . 
\end{equation}
\end{enumerate}
\end{cor}

This result is related to a question posed in \cite[Pg. 243]{PonceSpector}, where a similar result is obtained for functions belonging to $\mathrm{SBV}(\R^d)$ (these are bounded variation functions whose derivative has no Cantor part) and the authors ask whether or not their result can be extended to $\mathrm{BV}(\R^d)$.  

\subsection{Hyperuniform point processes}

Variants of formula \eqref{eq Davila} have recently been used in the study of the asymptotic behavior of linear statistics of hyperuniform point processes \cite{Lin24, LMOC,MMOC}. In this work we will further explore this interaction to prove results for functions of bounded variation, which follow from the theory of point processes, and obtain precise asymptotics for linear statistics from variants of formula \eqref{eq Davila}.

To introduce our results, we need to define some concepts from the theory of point processes. For further details and examples, see Section~\ref{Appl to PP} and the references cited therein.

A point process can be seen as a random locally finite point set. We are going to consider only simple point processes (no coincident points) that are invariant under rigid motions (stationary and isotropic). 

Of particular interest in the study of random point processes is the behavior of the so-called linear statistics, i.e. random variables
\[
\mathcal{X}(f)=\int_{\R^d} f \, d\mathcal{\mathcal{X}}=\sum_i f(x_i),
\]
where $x_i\in \R^d$ are the random points generated by the process. In particular, the number variance $\Var(\mathcal{X}(\chi_A))$, i.e., the fluctuation of the random variable counting the number of points of the process falling in the set $A\subset \R^d$, is of special importance.

As a consequence of Bochner’s theorem, see \cite[Theorem 1.5]{coste},
the variance of the linear statistic $\mathcal{X}(u)$ for $u$ smooth and compactly supported (and hence, by density, for any $u\in L^1(\R^d)\cap L^2(\R^d)$) is
\[
\Var(\mathcal{X}(u))=\int_{\R^d} |\widehat{u}(\omega)|^2 d\mathcal{S}(\omega).
\]
Here $\mathcal{S}$ is a positive measure called spectral or structure measure. For brevity, when $A$ is a set, we will sometimes denote $\mathcal{X}(A) := \mathcal{X}(\chi_A)$.

In most of the cases of interest in the literature, this measure is absolutely continuous with respect to the Lebesgue measure, and the function $\mathcalboondox{s}$ is non-negative, bounded and uniformly continuous 
(see Section~\ref{Appl to PP} for more details).

Our interest will focus on the class of hyperuniform point processes, that has recently received considerable attention, see the survey paper \cite{LacSurvey}.
A point process in $\R^d$ is called \emph{hyperuniform} when the fluctuation of the counting random variable on balls (the spherical number variance) is suppressed at large scales i.e.   
\[
\Var(\mathcal{X}(B(0,R)))=o(|B(0,R)|),\;R\to +\infty.
\]
Under very general conditions this is equivalent to $\mathcalboondox{s}(0)=0$.

Hyperuniform point processes whose structure function satisfies
\begin{equation}\label{local behavior structure function}
    \mathcalboondox{s}(\omega)=c|\omega|^{\alpha} + o(|\omega|^{\alpha}), \qquad \text{for } \omega\to  0,
\end{equation}
for some $\alpha, c>0$ are particularly relevant in condensed
matter physics and materials science \cite{Tor16}. To simplify notation we assume from now on that $c=1$ in (\ref{local behavior structure function}). In the seminal work \cite{Tor16} (see also \cite{torquato-stillinger, To18}), by using precise formulas for the variance and studying the asymptotic behavior of Bessel functions, Torquato proved that, depending on whether $\alpha>1$, $\alpha=1$ or $0<\alpha<1$ in \eqref{local behavior structure function}, the large scale asymptotic scaling of the number variance falls into three different regimes
\begin{description}
\item[Type I point processes]  In this case $\mathcalboondox{s}(\omega)=|\omega|^{\alpha} + o(|\omega|^{\alpha})$ for $\omega\to  0$, for some $\alpha>1,$
and then $\Var(\mathcal{X}(B(0,R)))=O(R^{d-1}),\,R\to +\infty.$
\medskip
\item[Type II point processes]  In this case $\mathcalboondox{s}(\omega)=|\omega|+ o(|\omega|)$ for $\omega\to  0,$ and then $\Var(\mathcal{X}(B(0,R)))=O(R^{d-1}\log R),\,R\to +\infty.$ 
\medskip
\item[Type III point processes] In this case $\mathcalboondox{s}(\omega)=|\omega|^{\alpha} + o(|\omega|^{\alpha})$ for $\omega\to  0$, for some $0<\alpha<1,$ and then $\Var(\mathcal{X}(B(0,R)))=O(R^{d-\alpha}),\,R\to +\infty.$
\end{description}

Our next result goes one step forward, extending Torquato's classification of the three asymptotic regimes above described to the variance of the linear statistic associated with a general bounded function of bounded variation.

\begin{teo}\label{asymptotic behavior}
Let $\mathcal{X}$ be a stationary hyperuniform and isotropic point process on $\R^d$, with structure function $\mathcalboondox{s}$ satisfying \eqref{local behavior structure function} and first intensity equal to one. Let $\mathcal{X}_L$ denote the point process dilated by $L$ so that the first intensity of $\mathcal{X}_L$ is equal to $L^d$ (i.e $\mathcal{X}_L$ is the process with points $x_i/L$ if $x_i$ are the points of $\mathcal{X}$). Then, for any $u\in \mathrm{BV}(\R^d)\cap L^\infty(\R^d)$,
\begin{enumerate}
    \item[i.)] If $\mathcal{X}$ is of type I ($\alpha>1$), then
    \begin{equation}\label{eq var Jump1}
    \lim_{L\to\infty} \frac{\Var(\mathcal{X}_L(u))}{L^{d-1}}=  \left( \int_0^\infty \frac{\mathcalboondox{s}(r)}{r^2} \,dr\right) \frac{\mathcal{J}(u)}{2\pi^2}.
    \end{equation}
    \item[ii.)] If $\mathcal{X}$ is of type II ($\alpha=1$), then
    \begin{equation}\label{eq var Jump2}
    \lim_{L\to\infty}\,\frac{\Var(\mathcal{X}_L(u))}{L^{d-1}\log L} = \frac{\mathcal{J}(u)}{2\pi^2}.
    \end{equation}
\item[iii.)]  If $\mathcal{X}$ is of type III ($0<\alpha<1$), then
    \begin{equation}\label{eq var Jump3}
    \lim_{L\to\infty} \frac{\Var(\mathcal{X}_L(u))}{L^{d-\alpha}} = \int_{\R^d} |\widehat{u}(\omega)|^2 |\omega|^\alpha \,d\omega.
    \end{equation}
\end{enumerate}    
\end{teo}
We point out that the integral in \eqref{eq var Jump1} is always finite by \eqref{local behavior structure function} and the fact that the structure function is bounded, and the integral in (\ref{eq var Jump3}) is always finite, because
$BV\cap L^{\infty}\subset H^{\alpha/2}$ for $0\le \alpha<1,$ see Lemma~\ref{lemma_inclusion}.
\begin{rem}
By applying the above theorem to $u=\chi_\Omega$, where $\Omega$ is a set of finite measure and finite perimeter, we get
\begin{equation*}
    \lim_{L\to\infty}\frac{\Var(\mathcal{X}(L\Omega))}{\Per(L\Omega)}=\frac{1}{2\pi^2}\left( \int_0^\infty \frac{\mathcalboondox{s}(r)}{r^2} \,dr\right) 
\end{equation*}
which shows that Torquato's result naturally extends from the number variance over spherical windows to the number variance associated to any such window $\Omega$.

\end{rem}

\begin{rem}
In the proof of this result one can see that if the function $u$ is rotationally invariant, as in the case of $u=\chi_{B(0,1)}$, which is the function one can use to recover the result of Torquato, it is also possible to consider non-isotropic point processes.
\end{rem}

\begin{rem}
If $\Omega\subset \R^d$ is of finite measure and
$\chi_\Omega\notin BV$, then the limsup of the left-hand side of \eqref{eq var Jump1} is infinity, but one can still get finite limits 
in (\ref{eq var Jump3}) with the parameter $\alpha$ chosen appropriately in relation  with the Minkowski dimension of the boundary.
For example, if $\Omega$ is the interior of the Koch snowflake in $\R^2$ the limit in (\ref{eq var Jump3}) is infinite except when
$0<\alpha<2-\log(4)/\log(3),$ in which case we obtain the corresponding Sobolev seminorm \cite[Corollary 1.4]{FR13}. See also \cite{Lin24} for a more in-depth treatment and  \cite{SodinWennmanYakir2023} for a similar result.

\end{rem}

For hyperuniform point processes in the Euclidean space our result above can be applied to generalize previous results about the behavior of linear statistics. For example, the following result for zeros of the Gaussian entire function (GEF) was first proved in \cite{FH99} for characteristic functions of sets of smooth boundary and finally in \cite{NS11} in a very general setting.

\begin{cor}\label{corolari_GEF}
    Let $\mathcal{X}$ be the point process given by the zeros in $\mathbb C$ of the GEF $$f(z)=\sum_{n=0}^\infty  \frac{a_n}{\sqrt{n!}} z^n,$$
    where $a_n$ are i.i.d complex normal random variables. Then, for every $u\in \mathrm{BV}(\R^2)\cap L^\infty(\R^2)$
    \[
    \lim_{L\to\infty}\frac{\Var (\mathcal{X}_L(u))}{L}=\frac{\zeta(3/2)}{8\pi^{3/2}} \mathcal{J}(u).
    \]
\end{cor}

Further examples are discussed in \cite{LMOC}. 

\subsection{BMO-type seminorms and the jump}

Formulas for the perimeter and, more generally, for the norm of bounded variation functions were developed in \cite{ABBF} with the original motivation of
characterizing the space of functions of bounded variation without the use of distributional derivatives. To further study the different parts of the norm 
of a function of bounded variation associated to the decomposition of $Du$, we consider a variation of the isotropic BMO-type semi-norms introduced by Ambrosio, Bourgain, Brezis, and Figalli in \cite{ABBF}
\[
\kappa_\varepsilon(u)=\varepsilon^{d-1} \sup _{\mathcal{Q}_\varepsilon } \sum_{Q^{\prime} \in \mathcal{Q}_\varepsilon} \fint_{Q^{\prime}}\left|u(x)-\fint_{Q^{\prime}} u(y)\, dy\right|\, dx,
\]
where the supremum is taken over all families $\mathcal{Q}_\varepsilon$ of disjoint cubes with side length $\varepsilon$, with arbitrary orientation and cardinality. If $u=\chi_A$, where $A \subset \mathbb{R}^d$ is a Caccioppoli set, they show that one has the convergence \cite[Eq. (4.4)]{ABBF}
\[
\lim _{\varepsilon \rightarrow 0} \kappa_\varepsilon\left(\chi_A\right)=\frac{1}{2}\left|D \chi_A\right|(\mathbb{R}^d).
\]

This result has been extended by Fusco, Moscariello and Sbordone (see \cite{FMS16}) to $u \in S B V\left(\mathbb{R}^d\right)$. In  \cite{FMS16} they prove that for such $u$ one has
\[
\lim _{\varepsilon \rightarrow 0} \kappa_\varepsilon(u)=\frac{1}{4} \int_{\mathbb{R}^d}|\nabla u|+\frac{1}{2}\left|D^s u\right|\left(\mathbb{R}^d\right).
\]

The original version was proved under an extra technical assumption on the function $u$, which was subsequently removed in \cite{dPFP17}. For the $L^2$ version see \cite{FMS18}.

Let $Q_0$ denote the unit cube $[-\frac12, \frac12)^d$. Given $t\in Q_0$ and $U\in O(d)$, let 
\[
\mathcal{Q}_\varepsilon (t,U)= \left\{\varepsilon U(Q_0+t+n): n\in\mathbb Z^d\right\}
\]
Given $u\in \mathrm{BV}(\R^d)\cap L^\infty(\R^d)$ we define the  BMO-type functional 
\[
\tau_\varepsilon(u)=\varepsilon^{d-1} \left\langle \sum_{Q'\in \mathcal{Q}_\varepsilon (t,U)} \fint_{Q'}\left|u(x)-\fint_{Q'} u(y) \,dy\right|^2 \,dx \right\rangle
\]
where for $f:Q_0\times O(d)\rightarrow \R$ integrable
\[
\langle f(t,U) \rangle := \int_{Q_0}\int_{\mathrm{O}(d)} f(t,U)\,dU\,dt,
\]
with $dU$ the Haar measure in $\mathrm{O}(d)$. For this functional we prove, with essential use of point process theory, that
\begin{teo}\label{BMO2 integral}
If $u\in \mathrm{BV}(\R^d)\cap L^\infty(\R^d)$, then 
\begin{equation}\label{eq Jump BMO}
\lim_{\varepsilon\to 0^+} \tau_\varepsilon(u)= C\, \mathcal{J}(u),
\end{equation} 
for
\begin{align*}
C=\frac{\Gamma(d/2)}{2\sqrt{\pi}\,\Gamma\!\left(\frac{d+1}{2}\right)}
      \int_{\mathbb{R}^d} |x| \prod_{j=1}^d (1 - |x_j|)_+ \, dx.
\end{align*}
In particular, if $\Omega\subset \R^d$ is a set with finite measure and finite perimeter, then 
\begin{equation}\label{eq perimetro BMO}
\lim_{\varepsilon\to 0^+} \tau_\varepsilon(\chi_\Omega)= C\, \Per(\Omega). 
\end{equation} 
\end{teo}

%On the other hand, combining results from \cite{FMS16}, and \cite{dPFP17}  (see Proposition (CITAR) below), we can prove that
%
%\[
%\tau^{\mathrm{sup}}_\varepsilon(u)=\varepsilon^{d-1} \sup _{\mathcal{Q}_\varepsilon }  \sum_{Q^{\prime} \in \mathcal{Q}_\varepsilon} \fint_{\varepsilon U(Q+n)}\left|u(x)-\fint_{\varepsilon U(Q+n)} u(y) dy\right|^2\, dx \leq \mathcal{J}(u).
%
%\]
%Therefore, from Theorem~\ref{BMO2 integral} we get the following corollary.
%
%
%
%\begin{teo}\label{BMO2 infinito}
%If $u\in \mathrm{BV}(\R^d)\cap L^\infty(\R^d)$, then
%\begin{equation}\label{eq perimetro con fourier 1}
%\lim_{\varepsilon\to 0^+} \tau^{\mathrm{sup}}_\varepsilon(u)= \mathcal{J}(u) 
%\end{equation} 
%\end{teo}

\section{From Corollary~\ref{char Davila} to Corollary~\ref{char Beretti}}\label{from davila}

This section is devoted to proving Claim~\ref{From Davila to BerGenn}. We remark that Corollary~\ref{char Beretti} may as well be obtained as a consequence of Theorem~\ref{Jump con diferencias} but we prefer to give here a direct and simple proof, ad hoc for characteristic functions, postponing the more technical proof of Theorem~\ref{Jump con diferencias} which applies to general $BV$ functions. What makes it possible to provide such a simplified proof is, essentially, the trivial observation that $|u(x)-u(y)|^2=|u(x)-u(y)|$ when $u$ is a characteristic function. This simple fact will allow to exploit the following lemma.

%We will prove that from the formula 
% \[
%  \lim_{L\to \infty} \iint_{\R^d\times\R^d}
% \frac{\bigl|\chi_{\Omega}(x)-\chi_{\Omega}(y)\bigr|}{|x-y|}\,\rho_L(x-y)\,dx\,dy=K_{d}\Per(\Omega),
% \]
% we can obtain 
% \[
% \lim_{R\to\infty} \frac {2\pi^2}{R} \int_{B(0,R)} |x|^2 |\widehat{\chi_\Omega}(x)|^2\, dx=\Per(\Omega).
% \]
%One of the main ingredients is the following lemma.
%which will be used also in subsequent generalizations.

\begin{lem}\label{divided}
Let $f\in L^1(\R^d)$ be an even function. If $u\in L^2(\R^d)$, then
\begin{equation}\label{eq divided}
\begin{split}
&\widehat{f}(0)\int_{\R^d} |\widehat{u}(\omega)|^2\, d\omega - \int_{\R^d} \widehat{f}(\omega)|\widehat{u}(\omega)|^2\, d\omega\\
&=\frac{1}{2}\iint_{\R^d\times\R^d}
\bigl|u(x)-u(y)\big|^2\,f(x-y)\,dx\,dy.
\end{split}
\end{equation}
\end{lem}

\begin{proof}
% By a standard density argument, it is enough to prove the result for a test function $u\in \mathcal{_c(\R^d)$. The left-hand side of \eqref{eq divided} is
% \begin{equation*}
% \int_{\R^d} |\widehat{u}(\omega)|^2 \widehat{\gamma}(\omega)\, d\omega=\widehat{f}(0)\int_{\R^d} |\widehat{u}(\omega)|^2\, d\omega - \int_{\R^d} \widehat{f}(\omega)|\widehat{u}(\omega)|^2\, d\omega.
% \end{equation*}
The first summand on the left  of \eqref{eq divided} can be rewritten as
\begin{align*}
\widehat{f}(0)\int_{\R^d} |\widehat{u}(\omega)|^2\, d\omega&=\widehat{f}(0)\int_{\R^d} |u(x)|^2\, dx=\iint_{\R^d\times\R^d}|u(x)|^2f(\xi)\, dx\, d\xi\\
&=\iint_{\R^d\times\R^d}|u(x)|^2f(y-x)\, dx\, dy\\
&=\iint_{\R^d\times\R^d}|u(x)|^2f(x-y)\, dx\, dy,
\end{align*}
where we used the definition of Fourier transform, Plancherel's theorem, the change of variable $\xi=y-x$, and the fact that $f$ is even.

The second summand on the left of \eqref{eq divided} coincides with the inner product $\langle \widehat{f}\, \widehat{u},\widehat{u}\rangle$, which by the properties of the Fourier transform is equivalent to the inner product $\langle u * f,u\rangle$. Namely, we have
\begin{equation*}
\int_{\R^d} \widehat{f}(\omega)|\widehat{u}(\omega)|^2\, d\omega=\iint_{\R^d\times\R^d} u(y)f(x-y)\overline{u(x)}\, dx\, dy.
\end{equation*}
Summing up, we got
\begin{align*}
&\widehat{f}(0)\int_{\R^d} |\widehat{u}(\omega)|^2\, d\omega - \int_{\R^d} \widehat{f}(\omega)|\widehat{u}(\omega)|^2\, d\omega\\
&=\iint_{\R^d\times\R^d} \left(|u(x)|^2-u(y)\overline{u(x)}\right)f(x-y)\, dx\, dy.
\end{align*}
 
The same formula can be rewritten by relabeling $x$ by $y$ and vice versa, and then changing the resulting term $f(y-x)$ back to $f(x-y)$, since $f$ is even. Therefore, we finally obtain that

\begin{align*}
&\widehat{f}(0)\int_{\R^d} |\widehat{u}(\omega)|^2\, d\omega - \int_{\R^d} \widehat{f}(\omega)|\widehat{u}(\omega)|^2\, d\omega \\
%=\frac{1}{2}\iint_{\R^d\times\R^d} (|u(x)|^2 -u(x)\overline{u(y)} -u(y)\overline{u(x)}+|u(y)|^2)  f(x-y)\,  dx\,dy\\
&=\frac{1}{2}\iint_{\R^d\times\R^d} |u(x)-u(y)|^2 \, f(x-y)\,  dx\,dy.
\end{align*}
\end{proof}

% Let $\eta\in  L^1(\R^d)$ be a symmetric function \(\eta(z)=\eta(-z)\) such that 
%\[
% \int_{R^d} \eta = 0.
%\]
% Using this function, for any  $u\in L^2(\R^d)$ we define the quadratic energy:
% \[
% E(u):=\frac12\iint_{\R^d\times\R^d}
% \bigl(u(x)-u(y)\bigr)^2\,\eta(x-y)\,dx\,dy.
% \]

% \begin{lem}~\label{divided1}
% For any $u\in L^2(\R^d)$
% \[ 
%   E(u) = -\int_{\R^d}\widehat\eta(\xi)\,|\widehat u(\xi)|^2\,d\xi  
% \]
% \end{lem}
% \begin{proof}

% Expanding the square
% \[
% E(u)=\tfrac12\iint\bigl(u(x)^2+u(y)^2-2u(x)u(y)\bigr)\eta(x-y)\,dx\,dy.
% \]

% The first two integrals coincide by the symmetry of \(\eta\) and the value is 0 because we assume that the 0-moment of $\eta$ vanishes. Thus,
% \[
% E(u)= -\iint u(x)u(y)\,\eta(x-y)\,dx\,dy = -\int_{\R^d} u(x)\,(\eta*u)(x)\,dx.
% \]
% Passing to Fourier:
% \[
% (\eta*u)^{\wedge}(\xi)=\widehat\eta(\xi)\,\widehat u(\xi).
% \]
% Therefore,
% \[
% \int u(x)\,(\eta*u)(x)\,dx
% = \int \widehat u(\xi)\,\overline{\widehat\eta(\xi)\,\widehat u(\xi)}\,d\xi
% = \int \widehat\eta(\xi)\,|\widehat u(\xi)|^2\,d\xi.
% \]

% We obtain the identity
% \[
% E(u)=
% \int_{\R^d}-\widehat\eta(\xi)\,|\widehat u(\xi)|^2\,d\xi.
% \]
% \end{proof} 

\begin{teo}\label{thm:smooth} 
Let $g\in C^\infty_c(\R)$ be a smooth, even function with compact support such that $g(0) = 0$ and 
\[
\int_0^\infty \frac{g(t)}{t^2} \, dt \ne 0.
\]
For any Caccioppoli set $\Omega\subset \R^d$ we have
\begin{equation}\label{smooth}
    \lim_{L\to \infty} L \int_{\R^d}  | \widehat{\chi_\Omega}|^2(\xi)\, g \left(\frac{|\xi|}{L}\right) d\xi = \frac{1}{2 \pi ^2}\left(\int_0^\infty \frac{g(t)}{t^2} \, dt\right) \Per(\Omega).
\end{equation}
\end{teo}

\begin{proof}
The fact that $g$ is smooth and even on $\R$ implies that $g(|\cdot|)$ is smooth on $\R^d$, besides being obviously even and compactly supported. Therefore, its Fourier transform decays faster than any polynomial at infinity, ensuring that the function
\[
\rho(x) := C |x| \int_{\R^d} g(|\xi|) e^{-2\pi i x \cdot \xi}\,  d\xi,
\]
is integrable. In particular, we can choose $C$ so that $\int_{\R^d} \rho(x)\, dx = 1$. Therefore

\begin{align}\label{comp}
C^{-1} = \int_{\R^d} |x|\widehat {g(|\cdot|)}(x)\, dx &=  \frac 1{2\pi}(-\Delta)^{1/2}g(0) \\
&= \frac{\Gamma\left((d+1)/2\right)}{2\pi^{(d+3)/2}}  \int_{\R^d} \frac{g(0)-g(|x|)}{|x|^{d+1}}\, dx \nonumber \\
&= 
-\frac{\Gamma\left((d+1)/2\right)}{\pi^{3/2}\Gamma\left(d/2\right)} \int_0^\infty \frac{g(t)}{t^2}\, dt \ne 0.\nonumber
\end{align}
Now, set $\rho_L(x)=L^d\rho(L x)$, so that
\begin{equation}\label{eq:basic}
    \begin{split}
&\iint_{\R^d\times\R^d}
\frac{\bigl|\chi_\Omega(x)-\chi_\Omega(y)\bigr|}{|x-y|}\,\rho_L(x-y)\,dx\,dy  \\
&=L \iint_{\R^d\times\R^d}
\bigl(\chi_\Omega(x)-\chi_\Omega(y)\bigr)^2\,L^d\frac{\rho((x-y)L)}{|x-y|L}\,dx\,dy \\
&= L \iint_{\R^d\times\R^d}
\bigl(\chi_\Omega(x)-\chi_\Omega(y)\bigr)^2\, \eta_L(x-y)\,dx\,dy,
\end{split}
\end{equation}
where $\eta(x) = \rho(x)/|x| = C \widehat{g(|\cdot|)}(x)$. Since 
\[
\int_{\R^d} \eta_L(x)\,dx=\widehat{\eta_L}(0)=\widehat{\eta}(0)=Cg(0)=0,
\]
we can apply Lemma~\ref{divided} to the function $f = \eta_L$, and we get
\begin{equation}
\begin{split}
\iint_{\R^d\times\R^d}
\bigl(\chi_\Omega(x)-\chi_\Omega(y)\bigr)^2\, \eta_L(x-y)\,dx\,dy &= -2 \int_{\R^d} |\widehat{\chi_\Omega}(\omega)|^2 \widehat \eta(\omega/L)\, d\omega \\
&=  -2C \int_{\R^d} |\widehat{\chi_\Omega}(\omega)|^2 g(|\omega|/L)\, d\omega.
\end{split}
\end{equation}
If $\Omega$ is a Caccioppoli set, since  $\rho$ is integrable and radial (by the rotation invariance of the Fourier transform), and $\rho_L(x)=L^d\rho(L x)$, by \eqref{eq Davila} we get
\begin{equation*}
\begin{split}
&\lim_{L \to \infty} L\int_{\R^d} |\widehat{\chi_\Omega}(\omega)|^2 g(|\omega|/L)\, d\omega\\
&=  -\frac{1}{2C}\lim_{L\to\infty} \iint_{\R^d\times\R^d}
\frac{\bigl|\chi_\Omega(x)-\chi_\Omega(y)\bigr|}{|x-y|}\,\rho_L(x-y)\,dx\,dy  \\
&= \frac{1}{2\pi ^2}\left(\int_0^\infty \frac{g(t)}{t^2} \, dt\right) \Per(\Omega).
\end{split}
\end{equation*}
\end{proof}

\textbf{Proof of the Claim.}
We cannot apply \eqref{smooth} directly with $g(t) = t^2\chi_{[-1,1]}(t)$, because $g$ is not smooth. But given any $\varepsilon>0$,
we can squeeze $f \le g \le h$, where $f$ and $h$ are smooth, even, non-negative functions with compact support, vanishing at the origin and such that $\int_0^\infty \frac{h(t)-f(t)}{t^2}\, dt\le \varepsilon$. Since all three functions are non-negative, we have
\begin{equation}\label{liminf f<g}
   \liminf_{R\to\infty}  R \int_{\R^d}|\widehat{\chi_\Omega}(\xi)|^2 f\left(\frac{|\xi|}{R}\right)\, d\xi\le \liminf_{R\to \infty} R \int_{\R^d} |\widehat{\chi_\Omega}(\xi)|^2 g\left(\frac{|\xi|}{R}\right)\, d\xi, 
\end{equation}
and 
\[
\limsup_{R\to\infty}  R \int_{\R^d}|\widehat{\chi_\Omega}(\xi)|^2 g\left(\frac{|\xi|}{R}\right)\, d\xi\le \limsup_{R\to \infty} R \int_{\R^d} |\widehat{\chi_\Omega}(\xi)|^2 h\left(\frac{|\xi|}{R}\right)\, d\xi.
\]
If $\Omega$ is a Caccioppoli set, we can apply Theorem~\ref{thm:smooth} to $f$ and $h$ and letting $\varepsilon\to 0$ we get
\[
\begin{split}
\lim_{R\to\infty} \frac{1}{R}\int_{B(0,R)} |\xi|^2 |\widehat{\chi_\Omega}(\xi)|^2\, d\xi& = \lim_{R\to\infty}  R \int_{\R^d}|\widehat{\chi_\Omega}(\xi)|^2 g\left(\frac{|\xi|}{R}\right)\, d\xi \\
&= \frac{\Per(\Omega)}{2\pi ^2} \int_0^\infty \frac{g(t)}{t^2}\, dt = \frac{\Per(\Omega)}{2\pi ^2}. 
\end{split}
\]\qed

\section{Spectral and nonlocal asymptotic formulas for the Jump}\label{Jump formulae}

Recall formula \eqref{eq BeGe} from Theorem~\ref{teo Beretti}, recently obtained by Beretti and Gennaioli, which provides an expression for the jump of
a bounded function $u:\R^d\to\R$ of bounded variation:
\begin{equation*}
\lim_{R\to\infty} \frac{1}{R} \int_{|\omega|\leq R} |\widehat{u}(\omega)|^2\,|\omega|^2\,d\omega = \frac{\mathcal{J}(u)}{2\pi^2}.
\end{equation*}
In this section we will show that there exists an entire one parameter family of formulas for the jump, to which formula \eqref{eq BeGe} belongs. We will obtain such a family of formulas as a consequence of the following general result.

\begin{pro}\label{global}
Let $C\in \R$ and $F:[0,\infty)\to\R$ belong to $L^1_{loc}([0,\infty))$. If, for some $0\le\alpha< 1$,
\begin{equation}\label{alpha<1 limit}
    \lim_{R \to \infty} \frac{1 - \alpha}{R^{\alpha - 1}} \int_R^\infty r^\alpha F(r) \, dr = C,
\end{equation}
or, for some $\alpha>1$,
\begin{equation}\label{alpha>1 limit}
    \lim_{R\to\infty}\frac{\alpha-1}{R^{\alpha-1}}\int_0^R r^{\alpha}\,F(r)\,dr = C,
\end{equation}
then \eqref{alpha<1 limit} holds for every $\alpha\in [0,1)$, \eqref{alpha>1 limit} holds for every $\alpha> 1$, and additionally
\begin{equation}\label{alpha=1 limit}
    \lim_{R\to\infty}\frac{1}{\log R}\int_0^R r\,F(r)\,dr = C.
\end{equation}

\end{pro}
\begin{proof}
Set $G_\alpha(R)=\int_{E_{\alpha,R}} r^\alpha F(r)\,dr$, where $E_{\alpha,R}=[0,R]$ when $\alpha\ge 1$ and $E_{\alpha,R}=[R,\infty)$ when $\alpha\in (0,1)$, and define $\varphi_\alpha(R)=|\alpha-1|G_\alpha(R)/R^{\alpha-1}$, for $\alpha\ne 1$ and $\varphi_1(R)=G_1(R)/\log R$.
% Then the condition $\lim_{R\to\infty}\varphi_\alpha(R)=C$ is equivalent to \eqref{alpha>1 limit} for $\alpha>1$ and to \eqref{alpha<1 limit} for $\alpha\in (0,1)$.

Let $\beta\ne 1$ be a nonnegative number, and assume that $\lim_{R\to\infty}\varphi_\beta(R)=C$. We want to show that this implies $\lim_{R\to\infty}\varphi_\alpha(R)=C$ for every $\alpha\ge 0$.

Observe that, for any $\alpha\ge 0$,
\begin{equation}\label{G alpha}
    G_\alpha(R)=\frac{\beta-1}{|\beta-1|}\int_{E_{\alpha,R}} r^{\alpha-\beta}\,G_\beta'(r)\,dr.
\end{equation}
For the sake of clarity, we now distinguish the two regimes $\alpha\in [0,1)$ and $\alpha\ge 1$ and we study them separately.

\textbf{Case 1: $\alpha\in [0,1)$.}
Integrating by parts, we can see that the integral in the right-hand side of \eqref{G alpha} is equivalent to
\[
\frac{\beta-1}{|\beta-1|}\left(\lim_{r\to\infty}r^{\alpha-\beta}\,G_\beta(r)-R^{\alpha-\beta}\,G_\beta(R)- (\alpha-\beta)\int_R^\infty r^{\alpha-\beta-1}\,G_\beta(r)\,dr\right).
\]
Since
\[
\lim_{r\to\infty}r^{\alpha-\beta}\,G_\beta(r)=\lim_{r\to\infty}\frac{r^{\alpha-1}}{|\beta-1|}\varphi_\beta(r)=\frac{C}{|\beta-1|}\lim_{r\to\infty}r^{\alpha-1}=0,
\]
multiplying by $|\alpha-1|/R^{\alpha-1}$ both sides of \eqref{G alpha} we get
\begin{equation}\label{phi alpha <1}
    \varphi_\alpha(R)=\frac{\alpha-1}{\beta-1}\,\varphi_\beta(R) + \frac{\alpha-\beta}{\beta-1}\,\frac{\alpha-1}{R^{\alpha-1}}\int_R^\infty r^{\alpha-2}\,\varphi_\beta(r)\,dr.
\end{equation}
By the L'Hopital's rule 
\[
\lim_{R\to\infty}\frac{\alpha-1}{R^{\alpha-1}}\int_R^\infty r^{\alpha-2}\,\varphi_\beta(r)\,dr=\lim_{R\to\infty}(-\varphi_\beta(R))=-C,
\]
so by taking limits in \eqref{phi alpha <1} we get
\[
\lim_{R\to\infty}\varphi_\alpha(R)=\frac{\alpha-1}{\beta-1}\,C- \frac{\alpha-\beta}{\beta-1}\,C=C.
\]

\textbf{Case 2: $\alpha\ge1$.}
In this case, again by integration by parts, the integral in the right-hand side of \eqref{G alpha} is equivalent to
\[
\frac{\beta-1}{|\beta-1|}\left(R^{\alpha-\beta}\,G_\beta(R) - \lim_{r\to 0} r^{\alpha-\beta}\,G_\beta(r)-(\alpha-\beta)\int_0^R r^{\alpha-\beta-1}\,G_\beta(r)\,dr\right).
\]
We claim that if $F\in L^1_{loc}([0,\infty))$, then $\lim_{r\to 0} r^{\alpha-\beta}\,G_\beta(r)=0$. Indeed, if $\beta>1$,
\[
|r^{\alpha-\beta}\,G_\beta(r)|=r^{\alpha-\beta}\left|\int_0^r s^\beta F(s) \, ds\right|\le r^\alpha\int_0^r |F(s)| \, ds \longrightarrow 0, \, \text{as }r\to 0.
\]
On the other hand, if $\beta\in [0,1)$, then $\alpha-\beta>0$, so that $\lim_{r\to 0}r^{\alpha-\beta}=0$, and it is enough to check that $G_\beta(r)$ converges to some finite constant. This actually happens, as one can see by writing
\[
G_\beta(r)=\int_r^\infty s^\beta F(s) \, ds= \int_r^R s^\beta F(s) \, ds+\int_R^\infty s^\beta F(s) \, ds,
\]
and observing that the first integral in the right hand-side is finite for any $R$ because $F\in L^1_{loc}([0,\infty))$, and that $R$ can be chosen large enough so that the second integral is also finite, since \eqref{alpha<1 limit} holds for $\alpha=\beta$.

Then the claim is proved. Now we distinguish two sub-cases. Suppose first that $\alpha>1$. Then, multiplying both sides of \eqref{G alpha} by $(|\alpha-1|)/R^{\alpha-1}$ we get
\begin{equation}\label{phi alpha>1}
    \varphi_\alpha(R)=\frac{\alpha-1}{\beta-1}\,\varphi_\beta(R) - \frac{\alpha-\beta}{\beta-1}\,\frac{\alpha-1}{R^{\alpha-1}}\int_0^R r^{\alpha-2}\,\varphi_\beta(r)\,dr.
\end{equation}
By the L'Hopital's rule 
\[
\lim_{R\to\infty}\frac{\alpha-1}{R^{\alpha-1}}\int_0^R r^{\alpha-2}\,\varphi_\beta(r)\,dr=\lim_{R\to\infty}\varphi_\beta(R)=C,
\]
so by taking limits in \eqref{phi alpha>1} we get
\[
\lim_{R\to\infty}\varphi_\alpha(R)=\frac{\alpha-1}{\beta-1}\,C- \frac{\alpha-\beta}{\beta-1}\,C=C.
\]

Suppose now that $\alpha=1$. Multiplying both sides of \eqref{G alpha} by $1/\log R$ we get
\begin{equation}\label{phi alpha=1}
    \varphi_1(R)=\frac{1}{\log R}\frac{\varphi_\beta(R)}{\beta-1} - \frac{1-\beta}{\beta-1}\,\frac{1}{\log R}\int_0^R \frac{\varphi_\beta(r)}{r}\,dr.
\end{equation}
By the L'Hopital's rule 
\[
\lim_{R\to\infty}\frac{1}{\log R}\int_0^R \frac{\varphi_\beta(r)}{r}\,dr=\lim_{R\to\infty}\varphi_\beta(R)=C,
\]
so by taking limits in \eqref{phi alpha=1} we get
\[
\lim_{R\to\infty}\varphi_1(R)=C.
\]
\end{proof}

\begin{rem}
    While both \eqref{alpha<1 limit} and \eqref{alpha>1 limit} are enough to imply \eqref{alpha=1 limit}, the opposite in general fails. Consider the following example.
\end{rem}
\noindent\textbf{Counterexample.}
Let
\[
F(r) = \begin{cases} \dfrac{1+\sin(\log r)}{r^2}, & r\ge 1,\\[4pt] 0, & r<1. \end{cases}
\]
Writing $T=\log R$,
\[
\int_0^R r\,F(r)\,dr
= \int_1^R \frac{1+\sin(\log r)}{r}\,dr
= \int_0^T (1+\sin u)\,du
= T - \cos T + 1,
\]
so
\[
\lim_{R\to\infty}\frac{1}{\log R}\int_0^R r\,F(r)\,dr
= \lim_{T\to\infty}1 - \frac{\cos T}{T} + \frac{1}{T}
= 1,
\]
hence \eqref{alpha=1 limit} holds. We show now that \eqref{alpha>1 limit} fails for $\alpha=2$, and, hence, \eqref{alpha>1 limit} fails for every $\alpha>1$ and \eqref{alpha=1 limit} fails for every $\alpha\in [0,1)$.

With the substitution $u=\log r$, we get
\[
\int_0^R r^2\,F(r)\,dr
= \int_1^R \bigl(1+\sin(\log r)\bigr)\,dr
= \int_0^T (1+\sin u)\,e^u\,du.
\]
Since $\displaystyle\int_0^T \sin(u)\,e^u\,du = \frac{e^T(\sin T - \cos T)+1}{2}$, we obtain
\[
\int_0^R r^2\,F(r)\,dr
= e^T + \frac{e^T(\sin T - \cos T)+1}{2}.
\]
Dividing by $R=e^T$,
\[
\frac{1}{R}\int_0^R r^2\,F(r)\,dr
= 1 + \frac{\sin T - \cos T}{2} + \frac{1}{2e^T},
\]
which oscillates between $1-\frac{\sqrt{2}}{2}$ and $1+\frac{\sqrt{2}}{2}$ and does not converge.

\medskip

\noindent Thanks to Proposition~\ref{global} we can generalize formula \eqref{eq BeGe} as follows.

\begin{pro}\label{formulas Jump1}
Let $u\in \mathrm{BV}(\R^d)\cap L^\infty(\R^d)$. Then,
\begin{itemize}
\item[(i)] $\displaystyle
\lim_{R\to\infty} \frac{1-\alpha}{R^{\alpha-1} }\int_{|\omega| \ge R} |\widehat u (\omega)|^2 |\omega|^\alpha\, d\omega =\frac{\mathcal{J}(u)}{2\pi^2}, \qquad \alpha \in [0,1)$.
\item[(ii)] 
$\displaystyle\lim_{R\to\infty} \frac{1}{\log R} \int_{|\omega| \le R} |\widehat u (\omega)|^2 |\omega|\, d\omega = \frac{\mathcal{J}(u)}{2\pi^2}$.
\item[(iii)] $\displaystyle
\lim_{R\to\infty} \frac{\alpha-1}{R^{\alpha-1}} \int_{|\omega| \le R} |\widehat u (\omega)|^2 |\omega|^{\alpha}\, d\omega =\frac{\mathcal{J}(u)}{2\pi^2}, \qquad \alpha>1$.

\end{itemize}
\end{pro}
\begin{proof}
Observe that (i), (ii) and (iii) are equivalent, respectively, to \eqref{alpha<1 limit}, \eqref{alpha=1 limit} and \eqref{alpha>1 limit} with $C=\mathcal{J}(u)/(2\pi^2)$ and
\[
F(r) = r^{d-1}\int_{\mathbb S^{d-1}} |\widehat{u}(r\omega)|^2 d\sigma_{d-1}(\omega),
\]
where $\sigma_{d-1}$ is the surface area measure on the unit sphere $\mathbb S^{d-1}$. Since (iii) holds for $\alpha=2$, as in this case it coincides with \eqref{eq BeGe} proved in \cite[Corollary 1.3]{BerGen2024}, then, by Proposition \eqref{global}, (i), (ii) and (iii) are automatically satisfied.
\end{proof}
 
% \begin{rem}
%     From (iii) above it follows easily that $\mathrm{BV}\cap L^\infty(\R^d)\subset H^{\alpha/2}(\R^d)$ for $0<\alpha<1$ and the inclusion is strict.
% \end{rem}
\begin{rem}
Note that, for $\alpha=0$ we recover

\begin{equation}\label{jump2}
\lim_{R\to\infty}  R \int_{|\omega|\geq R} |\widehat{u}(\omega)|^2\, d\omega=\frac{\mathcal{J}(u)}{2\pi^2},
\end{equation}
already obtained in \cite[Corollary 1.3]{BerGen2024},
which will be particularly important in the sequel. 
\end{rem}

\noindent\textbf{Proof of Theorem~\ref{Jump con Fourier}.} We will first prove i.). Let $\alpha>1$. Given $\varepsilon>0$, let $\eta \in (0,1)$ and $M\geq 1$ to be defined. Then, let
\begin{align*}
 &R \int_{\R^d} |\widehat{u}(\omega)|^2 \mathcalboondox{s}(|\omega|/R)\, d\omega =I_1(R)+I_2(R)+I_3(R)\\
 &=R\left( \int_{|\omega|\le\eta R}+\int_{\eta R\le|\omega|\le MR}+\int_{|\omega|\ge MR}|\widehat{u}(\omega)|^2 \mathcalboondox{s}(|\omega|/R)\, d\omega \right).   
\end{align*}

We will prove that $I_1(R)$ and $I_3(R)$ are small and $I_2(R)$ is close to the desired limit.

Let us start with $I_1(R)$. By choosing $\eta$ small enough so that, in the region $|\omega|\le \eta R$, $|\mathcalboondox{s}(|\omega|/R) - (|\omega|/R)^\alpha| \le(|\omega|/R)^\alpha$, by the triangular inequality, we get
\begin{equation*}
\begin{split}
        |I_1(R)|&\le \frac{2}{R^{\alpha-1}}\int _{|\omega|\leq \eta R}|\omega|^\alpha |\widehat{u}(\omega)|^2\,d\omega.
\end{split}
\end{equation*}
Now, choose $\eta$ small enough such that it also holds
\[
\eta^{\alpha-1}  \frac{\mathcal{J}(u) }{\pi^2(\alpha-1)}< \frac{\varepsilon}{6}.
\]
Then by another triangular inequality,
\begin{equation*}
\begin{split}
        |I_1(R)|\le 2\eta^{\alpha-1}\left| \frac{1}{(\eta R)^{\alpha-1} }\int _{|\omega|\leq \eta R}|\omega|^\alpha |\widehat{u}(\omega)|^2\,d\omega - \frac{\mathcal{J}(u) }{2\pi^2(\alpha-1)}\right|+\frac{\varepsilon}{6}
\end{split}
\end{equation*}

Once $\eta$ is fixed, small enough so that the above inequality holds, by Proposition~\ref{formulas Jump1}, it is possible to take $R_0$ so that for every $R\geq R_0$ 
\begin{equation*}
    \left| \frac{1}{(\eta R)^{\alpha-1} }\int _{|\omega|\leq \eta R}|\omega|^\alpha |\widehat{u}(\omega)|^2\,d\omega - \frac{\mathcal{J}(u) }{2\pi^2(\alpha-1)}\right| \leq \frac{\varepsilon}{12}.
\end{equation*}

Hence, for every $R\geq R_0$, we have that
\[
I_1(R)\le \eta^{\alpha-1}\frac{\varepsilon}{6}+\frac{\varepsilon}{6}<\frac{\varepsilon}{3}.
\]

Now we will study the third integral $I_3(R)$. 
Since $\mathcalboondox{s}$ is bounded, then by \eqref{jump2} we have
\[
|I_3(R)|\leq \|\mathcalboondox{s}\|_\infty R \int_{|\omega|\geq MR} |\widehat{u}(\omega)|^2\,d\omega \leq \frac{\|\mathcalboondox{s}\|_\infty}{M} \left(\frac{\mathcal{J}(u)}{2\pi^2} +\frac{\varepsilon}{6}\right).
\]
Thus, taking $M$ big enough, we get that $\displaystyle |I_3(R)|\leq \frac{\varepsilon}{3}$. 

Finally, let us study the middle integral $I_2(R)$. To begin with, we can assume that $\eta$ and $M$ have been chosen, respectively, small enough and big enough such that
\begin{equation}\label{first bound I_2}
   \left| I_2(R)-  \frac{\mathcal{J}(u)}{2\pi^2}  \int_0^\infty  \frac{\mathcalboondox{s}(t)}{t^2}\,dt \right| \leq \left| I_2(R)-  \frac{\mathcal{J}(u)}{2\pi^2}  \int_\eta^M \frac{\mathcalboondox{s}(t)}{t^2}\,dt \right|+\frac{\varepsilon}{6}. 
\end{equation}

Now, consider a partition of the interval $[\eta,M]$
\[
\eta=\alpha_0<\alpha_1<\ldots<\alpha_{n-1}<\alpha_n=M,
\]
and for each interval $[\alpha_j,\alpha_{j+1}]$ choose an element $r_j$ of that interval. The Riemann-Stieltjes sum corresponding to the integral
\[
\int_\eta^M \frac{\mathcalboondox{s}(t)}{t^2}\,dt,
\] 
is then given by
\[
 \sum_{j=0}^{n-1} \mathcalboondox{s}(r_j) \left(\frac{1}{\alpha_j}-\frac{1}{\alpha_{j+1}}\right).
\]
Hence, taking a partition fine enough, we get that
\begin{equation*}
    \left| I_2(R)-  \frac{\mathcal{J}(u)}{2\pi^2}  \int_\eta^M \frac{\mathcalboondox{s}(t)}{t^2}\,dt \right|\le\left| I_2(R)-  \frac{\mathcal{J}(u)}{2\pi^2}  \sum_{j=0}^{n-1} \mathcalboondox{s}(r_j) \left(\frac{1}{\alpha_j}-\frac{1}{\alpha_{j+1}}\right)\right|+\frac{\varepsilon}{12}. 
\end{equation*}
Moreover, since $\mathcalboondox{s}$ is uniformly continuous in $[\eta, M]$, we can choose the partition fine enough such that it also holds
\[
\left|I_2(R)- \Gamma(R) \right| \leq \frac{\varepsilon}{12},
\]
where
\begin{equation*}
\begin{split}
    \Gamma(&R)=R \sum_{j=0}^{n-1} \int_{\alpha_j R\le|\omega|\le\alpha_{j+1} R} |\hat{u}(\omega)|^2 \mathcalboondox{s}\left(r_j\right) \,d\omega\\
    &=\sum_{j=0}^{n-1} \mathcalboondox{s}\left(r_j\right) \left( \frac{\alpha_j R}{\alpha_j} \int_{|\omega|\ge\alpha_j R} |\hat{u}(\omega)|^2  \,d\omega - \frac{\alpha_{j+1} R}{\alpha_{j+1}} \int_{|\omega|\ge\alpha_{j+1} R} |\hat{u}(\omega)|^2  \,d\omega \right).
\end{split}
\end{equation*}
Now, with fixed $n$, and fixed $\eta$ and $M$, by \eqref{jump2} we can choose $R$ large enough such that for any $j=0, \ldots, n$
\begin{equation*}
    \left|\alpha_j R\int_{|\omega|\ge\alpha_j R} |\hat{u}(\omega)|^2  \,d\omega-\frac{\mathcal{J}(u)}{2\pi^2}\right|\le\frac{\varepsilon\eta}{24n\|\mathcalboondox{s}\|_\infty},
\end{equation*}
from which follows
\[
\left|\Gamma(R)-  \frac{\mathcal{J}(u)}{2\pi^2} \sum_{j=0}^{n-1} \mathcalboondox{s}\left(\frac{r_j}{R}\right) \left(\frac{1}{\alpha_j}-\frac{1}{\alpha_{j+1}}\right) \right|\leq \frac{\varepsilon}{12}.
\]
All together, 
\[
\left| I_2(R)-  \frac{\mathcal{J}(u)}{2\pi^2}  \int_0^\infty  \frac{\mathcalboondox{s}(t)}{t^2}\,dt \right| \leq \frac{\varepsilon}{3},
\]
as desired.

We will now prove ii.). Let $\alpha=1$. Given $\varepsilon>0$, choose $\eta>0$ small enough such that $|\mathcalboondox{s}(|\omega|/R) - |\omega|/R| \le\varepsilon|\omega|/R$ in the region $|\omega|<\eta R$. By the triangular inequality,
\begin{equation*}
\begin{split}
       &\left| \frac{R}{\log R} \int_{|\omega|< \eta R} |\widehat{u}(\omega)|^2 \mathcalboondox{s}(|\omega|/R) \, d\omega-\frac{\mathcal{J}(u)}{2\pi^2}\right|\\
        &\le \frac{R}{\log R}\left|\int _{|\omega|< \eta R}|\widehat{u}(\omega)|^2\left(\mathcalboondox{s}\left(\frac{|\omega|}{R}\right)-\left(\frac{|\omega|}{R}\right)\right)\,d\omega\right|\\
        &+\left|\frac{1}{\log R}\int _{|\omega|< \eta R}|\omega| |\widehat{u}(\omega)|^2\,d\omega-\frac{\mathcal{J}(u)}{2\pi^2}\right|\\
        &\le \frac{\varepsilon}{\log R}\left|\int _{|\omega|< \eta R}|\widehat{u}(\omega)|^2|\omega|\,d\omega\right|+\left|\frac{1}{\log R}\int _{|\omega|< \eta R}|\omega| |\widehat{u}(\omega)|^2\,d\omega-\frac{\mathcal{J}(u)}{2\pi^2}\right|
\end{split}
\end{equation*}
By Proposition~\ref{formulas Jump1}, as $R\to \infty$ the second summand in the right-hand side tends to zero and the first summand to $\varepsilon \mathcal{J}(u)/(2\pi^2)$. 

On the other hand,
\begin{equation*}
    \left|\frac{R}{\log R} \int_{|\omega|\ge \eta R} |\widehat{u}(\omega)|^2 \mathcalboondox{s}(|\omega|/R)\, d\omega\right|\le \|\mathcalboondox{s}\|_\infty\frac{\eta R}{\eta\log R} \int_{|\omega|\ge \eta R} |\widehat{u}(\omega)|^2 \, d\omega,
\end{equation*}
and the right-hand side tends to zero as $R\to \infty$ by \eqref{jump2}. Then,
\begin{equation*}
     \lim_{R\to +\infty}\left| \frac{R}{\log R} \int_{\R^d} |\widehat{u}(\omega)|^2 \mathcalboondox{s}(|\omega|/R) \, d\omega-\frac{\mathcal{J}(u)}{2\pi^2}\right|\le \frac{\mathcal{J}(u)}{2\pi^2}\varepsilon,
\end{equation*}
and letting $\varepsilon\to 0$ we get the desired result. \qed

\medskip

One may have expected, in analogy with Proposition~\ref{formulas Jump1}, to see also the case $\alpha\in [0,1)$ in the statement of Theorem~\ref{Jump con Fourier}. The reason why it is not there is that for such a range the (properly normalized) limit is not proportional only to the jump, but to the whole Fourier fractional norm, as precisely stated in the following proposition.

\begin{pro}\label{Jump con Fourier alpha<1}
Let $\mathcalboondox{s}:[0,\infty)\to\mathbb R$ be a bounded continuous function such that, for some $\alpha\in(0,1)$,
\[
\mathcalboondox{s}(t)=t^\alpha+o(t^\alpha)
\qquad\text{as }t\to0.
\]
If $u\in \mathrm{BV}(\mathbb R^d)\cap L^\infty(\mathbb R^d)$, then
\begin{equation}\label{eq Fourier Jump3}
    \lim_{R\to\infty}
R^\alpha
\int_{\mathbb R^d}
|\widehat u(\omega)|^2
\mathcalboondox{s}(|\omega|/R)\,d\omega
=
\int_{\mathbb R^d}
|\widehat u(\omega)|^2|\omega|^\alpha\,d\omega .
\end{equation}
\end{pro}
Observe that the right-hand side of \eqref{eq Fourier Jump3} is the Sobolev seminorm $H^{\alpha/2}$. Such a semi-norm is automatically finite for bounded $u$ of bounded variation and $\alpha\in(0,1)$. We couldn't find a simple explicit reference for such a property in the literature, hence we include it here with a self contained proof for the reader's convenience.

\begin{lem}\label{lemma_inclusion}
    For $0\le \alpha<1$ one has the strict inclusion $$\mathrm{BV}(\R^d)\cap L^\infty(\R^d)\subset  H^{\alpha/2}(\R^d).$$
\end{lem}

\begin{proof}
 By Gagliardo's characterization of functions in fractional Sobolev spaces, it is enough to see that for a bounded function $f$ of bounded variation,
\[
[f]_{H^{\alpha/2}}^2=\int_{\mathbb{R}^{d}}\int_{\mathbb{R}^{d}}
\frac{|f(x+y)-f(y)|^2}{|x|^{d+\alpha}}\,dx\,dy < +\infty.
\]

We can split the integral above according to whether $|x|>1$ or $|x|\le 1.$ For the first part, since
\[
\int_{\R^d} |f(x+y)-f(y)|^2 dy \le \int_{\R^d} 2|f(x+y)|^2 dy + \int_{\R^d} 2|f(y)|^2 dy = 4\|f\|_2^2,
\]
we get
\[
\int_{|x|>1} \left( \int_{\R^d} \frac{|f(x+y)-f(y)|^2}{|x|^{d+\alpha}} dy \right) dx \le 4\|f\|_2^2 \int_{|x|>1} \frac{1}{|x|^{d+\alpha}} dx,
\]
which is finite because $\alpha>0$ and $f\in L^2(\R^d)$.

For the other part of the integral, we can reason as follows. First, because $f \in L^\infty(\mathbb{R}^d)$, we have the trivial bound,
\[
\int_{\R^d} \frac{|f(x+y)-f(y)|^2}{|x|^{d+\alpha}} dy \le 2\Vert f\Vert_\infty \int_{\mathbb{R}^{d}}
\frac{|f(x+y)-f(y)|}{|x|^{d+\alpha}}\,dy.
\]
Moreover, by standard properties of functions of bounded variation\footnote{the inequality is trivially true through the fundamental theorem of calculus for smooth functions, and extends to BV functions by a classical approximation argument, see, e.g., \cite[Chapter 5]{EvGar}}, we have \[\|f(\cdot + x) - f\|_{L^1} \le |x| |Df|(\mathbb{R}^d).
\]
It follows, 
\[
\int_{|x|\le 1} \left( \int_{\R^d} \frac{|f(x+y)-f(y)|^2}{|x|^{d+\alpha}} dy \right) dx\le 2\Vert f\Vert_\infty\Vert f\Vert_{BV} \int_{|x|\le 1}  \frac{1}{|x|^{d+\alpha-1}} \, dx,
\]
which is finite since $\alpha<1$.
\end{proof}

\noindent\textbf{Proof of Proposition~\ref{Jump con Fourier alpha<1}.}
By Lemma~\ref{lemma_inclusion},
\[
A_\alpha(u)
:=
\int_{\mathbb R^d}
|\widehat u(\omega)|^2|\omega|^\alpha\,d\omega
<+\infty,
\]
since it is equivalent to the Gagliardo seminorm of $u$.
Let $\varepsilon>0$. Since
\[
\mathcalboondox{s}(t)=t^\alpha+o(t^\alpha)
\qquad\text{as }t\to0,
\]
there exists $\eta>0$ such that, for every $t\in[0,\eta]$,
\begin{equation}\label{small t estimate}
  \left|\mathcalboondox{s}(t)-t^\alpha\right|
\le
\varepsilon t^\alpha .  
\end{equation}
By triangular inequality,
\[
\begin{split}
&\left|
R^\alpha
\int_{\mathbb R^d}
|\widehat u(\omega)|^2
\mathcalboondox{s}(|\omega|/R)\,d\omega
-
A_\alpha(u)
\right| \\
&\le
\left|
I_1(R)
-
\int_{|\omega|\le \eta R}
|\widehat u(\omega)|^2|\omega|^\alpha\,d\omega
\right|
+
|I_2(R)|
+
\int_{|\omega|>\eta R}
|\widehat u(\omega)|^2|\omega|^\alpha\,d\omega,
\end{split}
\]
where
\[
I_1(R)
:=
R^\alpha
\int_{|\omega|\le \eta R}
|\widehat u(\omega)|^2
\mathcalboondox{s}(|\omega|/R)\,d\omega
\]
and
\[
I_2(R)
:=
R^\alpha
\int_{|\omega|> \eta R}
|\widehat u(\omega)|^2
\mathcalboondox{s}(|\omega|/R)\,d\omega .
\]
Now, observe that the contribution of $I_2(R)$ vanishes in the limit as $R\to\infty$. Indeed, since
$\mathcalboondox{s}$ is bounded, we have
\[
|I_2(R)|
\le
\|\mathcalboondox{s}\|_\infty
R^\alpha
\int_{|\omega|>\eta R}
|\widehat u(\omega)|^2\,d\omega .
\]
By \eqref{jump2}, we know that
\[
\int_{|\omega|>\eta R}
|\widehat u(\omega)|^2\,d\omega
=
O\left(\frac1R\right),
\qquad\text{as }R\to\infty .
\]
Therefore $|I_2(R)|
\le
C R^{\alpha-1}$, which vanishes as $R\to\infty$ because $\alpha<1$.

Moreover, since $A_\alpha(u)<+\infty$, we also have
\[
\lim_{R\to\infty}\int_{|\omega|>\eta R}
|\widehat u(\omega)|^2|\omega|^\alpha\,d\omega= 0.
\]
Finally, on the region
$|\omega|\le \eta R$, we have
\[
\left|
R^\alpha \mathcalboondox{s}(|\omega|/R)-|\omega|^\alpha
\right|
=
R^\alpha
\left|
\mathcalboondox{s}(|\omega|/R)-(|\omega|/R)^\alpha
\right|
\le
\varepsilon |\omega|^\alpha .
\]
Consequently,
\[
\begin{split}
\left|
I_1(R)
-
\int_{|\omega|\le \eta R}
|\widehat u(\omega)|^2|\omega|^\alpha\,d\omega
\right|
&\le
\varepsilon
\int_{|\omega|\le \eta R}
|\widehat u(\omega)|^2|\omega|^\alpha\,d\omega \\
&\le
\varepsilon A_\alpha(u),
\end{split}
\]
and since $\varepsilon>0$ is arbitrary, the result follows. \qed

\medskip

% Observe that in the previous proof, if $u$ is more regular, we have with the same proof (considering $\alpha=2$) the following result
% 
% \begin{teo}\label{teo smooth}
% Let $\mathcalboondox{s}:\R\to\R$ be a bounded even continuous function such that
% \begin{equation}\label{Taylor 2 suave}
% \mathcalboondox{s}(t)=t^{2} + o(t^{2}).
% \end{equation}
% If $u\in W^{1,2}$, then
%
%\[
% \lim_{L\to\infty}L^{2} \int_{\R   ^d} |\widehat{u}(\omega)|^2 \mathcalboondox{s}(|\omega|/L)\, d\omega=\int_{\R^d} |\nabla u|^2\, dx.
%\]
% \end{teo}

We end this section by proving
Theorem~\ref{Jump con diferencias}, which provides a family of nonlocal asymptotic formulas for the Jump.

\medskip

\noindent\textbf{Proof of Theorem~\ref{Jump con diferencias}.}
Applying Lemma~\ref{divided} to the function $f = \rho_L$,
\begin{equation*}
\begin{split}
     &\iint_{\R^d\times\R^d}
\bigl(u(x)-u(y)\bigr)^2\,\rho_L(x-y)\,dx\,dy=2\int_{\R^d}(\widehat{\rho_L}(0)- \widehat{\rho_L}(\omega)) |\widehat{u}(\omega)|^2\, d\omega\\
&=2\int_{\R^d}(\widehat{\rho}(0)- \widehat{\rho}(\omega/L)) |\widehat{u}(\omega)|^2\, d\omega=2\int_{\R^d}\mathcalboondox{s}(|\omega|/L)|\widehat{u}(\omega)|^2\, d\omega,
\end{split}
\end{equation*}
% \begin{align*}
% &\widehat{\rho}(0)\int_{\R^d} |\widehat{u}(\omega)|^2\, d\omega - \int_{\R^d} \widehat{\rho}(\omega)|\widehat{u}(\omega)|^2\, d\omega\\
% &=\frac{1}{2}\iint_{\R^d\times\R^d} 
% \bigl|u(x)-u(y)\big|^2\,\rho(x-y)\,dx\,dy.
% \end{align*}
% As $\rho$ is radial,  we recall that there is a $\eta$ such that 
% $\widehat{\rho}(0)-\widehat{\rho}(w)=\eta(0)-\eta(|w|)$.
where $\mathcalboondox{s}(t):=g(0)-g(t)$. Since the function $\mathcalboondox{s}$ satisfies the hypothesis of Theorem~\ref{Jump con Fourier}, if $\alpha>1$, we get
\[
\begin{split}
\lim_{L\to\infty}& L\iint_{\R^d\times\R^d}
\bigl(u(x)-u(y)\bigr)^2\,\rho_L(x-y)\,dx\,dy\\
&=\, \frac{\mathcal{J}(u)}{\pi^2}\, \left( \int_0^\infty \frac{g(0)-g(t)}{t^2} \,dt\right),
\end{split}
\]
and if $\alpha=1$, 
\[
\lim_{L\to\infty}\frac{L}{\log L}\iint_{\R^d\times\R^d}
\bigl(u(x)-u(y)\bigr)^2\,\rho_L(x-y)\,dx\,dy = \frac{\mathcal{J}(u)}{\pi^2}  .
\]\qed

\section{Applications to hyperuniform point processes}\label{Appl to PP}

A point process in $\R^d$ is a measurable map 
\[\mathcal{X}:\Omega \longrightarrow M_p(\R^d)\]
where $(\Omega, \mathcal F,\mathbb P)$ is a probability space
and $M_p(\R^d)$ is the space of locally finite point measures with the 
$\sigma-$algebra generated by the evaluations and it can be 
also seen as a random locally finite point set, see \cite{HoKrPeVi}.

We are going to consider only simple point processes (no coincident points) which are invariant through rigid motions 
(stationary and isotropic). For stationary point processes, according to Campbell's formula,
\[
\mathbb E[\mathcal{X}(f)]=\rho_1 \int_{\R^{d}} f(x)\,dx,
\]
namely, the first intensity measure is a constant multiple of the Lebesgue measure, and the constant $\rho_1\in \R$ is called the first intensity. On the other hand, given two test functions $u,v\in \mathcal{C}_c(\mathbb R^d)$, the stationarity implies that
\[
\Cov(\mathcal{X}(u), \mathcal{X}(v))=\int_{\R^d} \left(\int_{\R^d} u(x)v(x+y)  d\mathcal{C}_0(y)  \right)\, dx,
\]
where $\mathcal{C}_0$ is a positive definite measure, usually called \emph{the reduced covariance measure} (see \cite[Cor.~10.4.IV]{MR950166}). This measure $\mathcal{C}_0$ is also invariant under rotations, and it can be written in terms of the first and second intensity functions as follows:
\[
\mathcal{C}_0=\rho_1 \delta_{0} + (\rho_2(0,|x|)-\rho_1^2) \,dx,
\]
where $\delta_0$ denotes the Dirac delta measure at the point $x=0$. The function 
\[
\rho_2^T(0,|x|)=\rho_2(0,|x|)-\rho_1^2
\]
is called \emph{two point truncated correlation function}. Since our point process is stationary and isotropic, the function $\rho_2^T$ depends only on $|x|$. So, from now on we will assume that it is a function of one real variable and therefore
\[
\mathcal{C}_0=\rho_1 \delta_{0} +\rho_2^T(|x|)\, dx.
\]
Moreover, we will assume that $\rho_2^T$ induces an integrable function in $\R^d$. The Fourier transform of the measure $\mathcal{C}_0$ is a positive measure $\mathcal{S}$, which has different names in the literature: structure measure, Bartlett spectrum, among others. In terms of this measure, we have that
\[
\Var(\mathcal{X}(u))=\int_{\R^d} |\widehat{u}(\omega)|^2 d\mathcal{S}(\omega).
\]
Since $\rho_2^T$ is integrable, the measure $\mathcal{S}$ is absolutely continuous with respect to the Lebesgue measure 
\[
\mathcal{S}=\rho_1 \,\mathcalboondox{s}(|\omega|)\, d\omega,
\]
where the function $\mathcalboondox{s}:\R\to [0,+\infty)$ is given by
\begin{equation}\label{s in terms of correlation function}
    \displaystyle\mathcalboondox{s}(|\omega|)=1+\frac{\widehat{\rho_2^T}(|\omega|)}{\rho_1}.
\end{equation}
We will call this function \emph{structure function}. Note that $\mathcalboondox{s}$ is uniformly continuous and satisfies 
\[
\lim_{t\to\infty} \mathcalboondox{s}(t)=1.
\]
In particular, $\mathcalboondox{s}$ is bounded. Using the structure function, the variance of the linear statistic $\mathcal{X}(u)$ can be written as follows
\begin{equation}\label{eq variance con s}
\Var(\mathcal{X}(u))=\int_{\R^d} |\widehat{u}(\omega)|^2 \,\mathcalboondox{s}(|\omega|)\, d\omega.
\end{equation}

Finally, we will assume that our point process is \emph{hyperuniform}, which is equivalent to saying that $\mathcalboondox{s}(0)=0$. Actually, we will assume a bit stronger assumption, namely that near the origin,
\begin{equation}
\mathcalboondox{s}(t)=c t^\alpha+o(t^\alpha),
\qquad t\to 0,
\end{equation}
for some $c>0$ and $\alpha>0$. For $\alpha>1$, this local behavior, and
the fact that  \emph{$\mathcalboondox{s}(t)$} is bounded imply that

\begin{equation}
\int_0^\infty \frac{\mathcalboondox{s}(t)}{t^2} \,dt<\infty.
\end{equation}
Note that another way to say that $\mathcalboondox{s}(0)=0$ is to say that the \emph{sum charge rule} holds
\begin{equation}\label{eq SCR}
\rho_1+\int_{\R^d} \rho_2^T(|x|)\,dx =0.
\end{equation}

The interested reader is referred to \cite{LacSurvey,coste}, and \cite{MR950166} for more details, and the proofs of the facts presented above. 

\medskip

\textbf{Proof of Theorem~\ref{asymptotic behavior}.} Clearly $\mathrm{BV}(\R^d)\cap L^{\infty}(\R^d)\subset L^1(\R^d)\cap L^2(\R^d)$. Then, for any function $\mathcalboondox{s}$ as in the hypothesis 
the following expression for the variance holds 
\[
\Var(\mathcal{X}(u))=\int_{\R^d} |\widehat{u}(\omega)|^2 \mathcalboondox{s}(\omega)d\omega,
\]
and for the rescaled point process $\mathcal{X}_L$, since $\widehat{u(\cdot/L)}(\omega)=L^d\widehat{u}(L\omega)$, we have
\[
\Var(\mathcal{X}_L(u))=L^d \int_{\R^d} |\widehat{u}(\omega)|^2 \mathcalboondox{s}(|\omega|/L)d\omega.
\]
Then, applying Proposition~\ref{Jump con Fourier alpha<1} in the case for $0<\alpha<1$, and Theorem~\ref{Jump con Fourier}, for $\alpha\ge 1$, gives the result. \qed

% \vskip 2cm

% When $\Omega\subset \R^d$ is a Caccipoli set, then by the De Giorgi structure theorem $\chi_\Omega\in BV$ and all the limits in the theorem above are finite. When $\chi_\Omega\not\in BV,$ and $\Omega$ is not Caccioppoli, the limits (\ref{eq var Jump1}), (\ref{eq var Jump2}) equal infinity. Indeed, our point processes satisfy that
% $\mathcalboondox{s}(|\omega|)=1-\hat{f}(\omega)$ for some $f\in L^1(\R^d)$ and therefore applying Lemma~\ref{divided} (see the proof of Proposition~\ref{prop BV})
% \begin{align*}
% \liminf_{L\to +\infty}\iint_{\R^d\times\R^d}
% \frac{\bigl|\chi_{\Omega}(x)-\chi_{\Omega}(y)\bigr|}{|x-y|}\,\rho_L(x-y)\,dx\,dy=+\infty    
% \end{align*}
% \[
% \liminf_{L\to +\infty}\iint_{\R^d\times\R^d}
% \frac{\bigl|\chi_{\Omega}(x)-\chi_{\Omega}(y)\bigr|}{|x-y|}\,\rho_L(x-y)\,dx\,dy=+\infty
% \]
% by \cite[Theorem 3']{BBM} and the claim follows from the computation in the proof of Theorem~\ref{thm:smooth}. But even in this case iii.) may apply ...

\medskip

\noindent\textbf{Proof of Corollary~\ref{corolari_GEF}.} 
A closed formula for the structure function $\mathcalboondox{s}$ of the planar GEF has been explicitly computed by Nazarov and Sodin in \cite{NS11}. In fact, they proved that for every $u\in L^1(\R^2)\cap L^2(\R^2)$ and every $L>0$, the number variance of the process is given by
\begin{equation}\label{eq:var-L-M}
\frac{\Var(\mathcal X_L(u))}{L}
=
L\int_{\R^2} |\widehat{u}(\omega)|^2\, \mathcalboondox{s}(|\omega|/L)\,d\omega,
\end{equation}
where
\[
\mathcalboondox{s}(t)=\pi^3 t^4 \sum_{n=1}^\infty \frac{1}{n^3} e^{-\pi^2 t^2/n},
\qquad t\ge 0.
\]

Clearly \eqref{eq:var-L-M} in particular holds for $u\in BV(\R^2)\cap L^\infty(\R^2)$ since this space is contained in $ L^1(\R^2)\cap L^2(\R^2)$. 

We claim that the function $\mathcalboondox{s}$ satisfies the assumptions of Theorem~\ref{asymptotic behavior} with
$\alpha=4$. The claim then gives
\begin{equation}\label{eq:M-limit}
   \lim_{L\to\infty}\frac{\Var(\mathcal X_L(u))}{L}
=
\left(\int_0^\infty \frac{\mathcalboondox{s}(t)}{t^2}\,dt\right)\frac{\mathcal J(u)}{2\pi^2}. 
\end{equation}
The integral on the right hand side can be computed explicitly. Indeed, since the summands are nonnegative, Tonelli's theorem yields
\[
\int_0^\infty \frac{\mathcalboondox{s}(t)}{t^2}\,dt
=
\pi^3 \sum_{n=1}^\infty \frac1{n^3}
\int_0^\infty t^2 e^{-\pi^2 t^2/n}\,dt.
\]
Using the standard identity
\[
\int_0^\infty t^2 e^{-a t^2}\,dt=\frac{\sqrt{\pi}}{4a^{3/2}},
\qquad a>0,
\]
with $a=\pi^2/n$, we get
\[
\int_0^\infty \frac{\mathcalboondox{s}(t)}{t^2}\,dt
=
\pi^3 \sum_{n=1}^\infty \frac1{n^3}\cdot \frac{n^{3/2}}{4\pi^{5/2}}
=
\frac{\sqrt{\pi}}{4}\sum_{n=1}^\infty \frac1{n^{3/2}}
=
\frac{\sqrt{\pi}}{4}\,\zeta(3/2).
\]
Substituting into \eqref{eq:M-limit} gives the desired result,
\[
\lim_{L\to\infty}\frac{\Var(\mathcal X_L(u))}{L}
=
\frac{\zeta(3/2)}{8\pi^{3/2}}\,\mathcal J(u).
\]
It only remains to prove the claim. First of all, $\mathcalboondox{s}$ is continuous, because the series appearing in the definition of $M$ is uniformly convergent on compact sets. Moreover, dominated convergence yields
\[
\lim_{t\to 0}\frac{\mathcalboondox{s}(t)}{t^4}
=
\pi^3 \sum_{n=1}^\infty \frac1{n^3}
=
\pi^3\zeta(3).
\]
Therefore
\begin{equation}\label{eq:M-origin}
\mathcalboondox{s}(t)=\pi^3\zeta(3)t^4+o(t^4)
\qquad\text{as }t\to 0.
\end{equation}
To see that $\mathcalboondox{s}$ is bounded on $[0,\infty)$ we can reason as follows.

For $0\le t\le 1$, boundedness is immediate. 
For $t\ge 1$, write
\[
\mathcalboondox{s}(t)
=
\pi^3 t^4 \sum_{n\le t^2}\frac{1}{n^3}e^{-\pi^2 t^2/n}
+
\pi^3 t^4 \sum_{n>t^2}\frac{1}{n^3}e^{-\pi^2 t^2/n}
=:I_1(t)+I_2(t).
\]
For $I_2(t)$, using $e^{-x}\le 1$,
\[
I_2(t) \le \pi^3 t^4 \sum_{n>t^2}\frac{1}{n^3}\lesssim \pi^3 t^4 \int_{t^2}^\infty \frac{1}{x^3}\, dx=\frac{\pi^3}{2}.
\]

For $I_1(t)$, using $e^{-x}\lesssim x^{-3}$ for $x>0$,
\[
e^{-\pi^2 t^2/n}\lesssim \Big(\frac{n}{t^2}\Big)^3,
\]
so that
\[
I_1(t)
\lesssim
t^4 \sum_{n\le t^2}\frac{1}{n^3}\Big(\frac{n}{t^2}\Big)^3
=
t^4 \cdot \frac{1}{t^6} \sum_{n\le t^2}1
= 1.
\]
Thus $\mathcalboondox{s}$ is bounded on $[0,\infty)$.
\qed

\begin{rem}
    In order to apply Theorem~\ref{asymptotic behavior} to derive the asymptotic behavior of the variance of the linear statistic associated with a bounded function of bounded variation of a hyperuniform point process, it is not necessary to have a closed formula for the 
    structure function. It is enough to use the first and second intensity functions instead, see \cite{Hannay} for the GEF.
 \end{rem}

\section{BMO type semi-norms}

In this section we prove our final main result, Theorem~\ref{BMO2 integral}.

Let $Q_0$ denote the unit cube
$
[-\frac12, \frac12)^d.
$ Given $t\in Q_0$ and $U\in O(d)$ let 
\[\mathcal{Q}_\varepsilon (t,U)=\{ \varepsilon  U ( Q_0+t+n)\;:\; n\in \mathbb Z^d \}.\]
Given $u\in \mathrm{BV}(\R^d)\cap L^\infty(\R^d)$ we consider the  BMO-type functional 
\[
\tau_\varepsilon(u)=\varepsilon^{d-1} \left\langle \sum_{Q'\in \mathcal{Q}_\varepsilon (t,U)} \fint_{Q'}\left|u(x)-\fint_{Q'} u(y)\, dy\right|^2\, dx \right\rangle
\]
where for $f:Q_0\times O(d)\rightarrow \R$ integrable
\[\langle f(t,U) \rangle=\int_{Q_0}\int_{\mathrm{O}(d)} f(t,U)  dU\,dt,\]
with $dU$ the Haar measure in $\mathrm{O}(d)$.

% \begin{teo}\label{BMO Jump}
% If $u\in \mathrm{BV}(\R^d)\cap L^\infty(\R^d)$, then there is a universal constant $C$ such that 
% \begin{equation}\label{formula_tau}
% \lim_{\varepsilon\to 0^+} \tau_\varepsilon(u)= C \mathcal{J}(u). 
% \end{equation} 
% \end{teo}

We will prove our last result using a probabilistic argument. 

\noindent\textbf{Proof of Theorem~\ref{BMO2 integral}.} 
Consider the stationary point process
\[
\mathcal{X}=\sum_{n\in\Z^d} \delta_{(n+X+X_n)}, 
\]
where $X$ is a random variable uniformly distributed in the cube $Q_0$, and the random variables $X_n$ are i.i.d. with a distribution law absolutely continuous with respect to the Lebesgue measure, $\phi(x)\,dx$. For the rescaled process $\mathcal{X}_L$, which we recall to be the process $\mathcal{X}$ dilated so that the first intensity is equal to $L^d$, for any $u \in L^1(\R^d) \cap L^2(\R^d)$ (hence in particular for $u\in \mathrm{BV}(\R^d)\cap L^\infty(\R^d)$) it holds \cite{Ya21}
\begin{align*}
&\Var(\mathcal{X}_L(u))\\
&=L^d \int_{\R^d}|\widehat{u}(\omega)|^2(1-|\widehat{\phi}(\omega/L)|^2)\, d\omega+L^{2 d} \sum_{m \in \mathbb{Z}^d \setminus\{0\}} |\widehat{\phi}(m)|^2\cdot|\widehat{u}(L m)|^2.
\end{align*}
From now on, we will assume that $\phi(x)=\chi_{[-1/2,1/2)^d}(x)$, and hence $\widehat{\phi}(m)=0$ for every integer $m\neq 0$.  Therefore, from the previous result, we get

\begin{align*}
\Var(\mathcal{X}_L(u))&=L^d \int_{\R^d}|\widehat{u}(\omega)|^2(1-|\widehat{\phi}(\omega/L)|^2)\, d\omega.
\end{align*}
Now, if we consider $U\in \mathrm{O}(d)$ and define
\[
\mathcal{X}_{L,U}=\sum_{n\in\Z^d} \delta_{\frac{1}{L}U(n+X+X_n)}, 
\]
then
\begin{align}
\Var(\mathcal{X}_{L,U}(u))&=L^d \int_{\R^d}|\widehat{u}(\omega)|^2(1-|\widehat{\phi}(U(\omega/L))|^2)\, d\omega.\label{eq varianza1}
\end{align}
Now the idea is to consider the random point process $\mathcal{X}_{L,U}$, defined as above, but in which not only $X$ and the $X_n$ are random variables, but also $U$ is a uniformly distributed random variable on $\mathrm{O}(d)$, and compute the variance $\Var(\mathcal{X}_{L,U}(u))$ in two different ways. Recall that, by the total variance law (see for instance \cite[Ch. 4]{Du}, we have that

\begin{equation}\label{eq totaldesvariation}
\Var(\mathcal{X}_{L,U}(u))=\E\big(\Var(\mathcal{X}_{L,U}(u)|U)\big)+\Var\big(\E(\mathcal{X}_{L,U}(u)|U)\big).
\end{equation}
First of all, since the first intensity function of $\mathcal{X}_L$ is a constant equal to $L^d$, and the Lebesgue measure is invariant under the action of $\mathrm{O}(d)$ on $\R^d$, we obtain that

\begin{align*}
\E(\mathcal{X}_{L,U}(u)|U)&= L^d \int_{\R^d} u(x) \,dx. 
\end{align*}
In consequence, the second term in \eqref{eq totaldesvariation} is equal to zero, and we have that 

\begin{equation}\label{eq simplification}
\Var(\mathcal{X}_{L,U}(u))=\E\big(\Var(\mathcal{X}_{L,U}(u)|U)\big).
\end{equation}
On the one hand, by \eqref{eq varianza1} we have that
\[
\E\big(\Var(\mathcal{X}_{L,U}(u)|U)\big)=L^d \int_{\mathrm{O}(d)} \int_{\R^d}|\widehat{u}(\omega)|^2(1-|\widehat{\phi}(U\omega/L)|^2)\, d\omega \,dU.
\]
If we define the radial function
\[
\mathcalboondox{s}(\omega)=\int_{\mathrm{O}(d)} (1-|\widehat{\phi}(U\omega)|^2) \,dU
\]
by Fubini's theorem, we get that
\[
\E\big(\Var(\mathcal{X}_{L,U}(u)|U)\big)=L^d  \int_{\R^d}|\widehat{u}(\omega)|^2 \mathcalboondox{s}(\omega/L)\, d\omega.
\]
On the other hand, again by the total variance law, but this time conditioning with respect to the variable $X$ uniformly distributed in $Q_0$, we obtain that

\begin{align*}
\Var(\mathcal{X}_{L,U}(u)|U)&= \E\big(\Var(\mathcal{X}_{L,U}(u)|U,t))+\Var\big(\E(\mathcal{X}_{L,U}(u)|U,t)\big)
\end{align*}

Note that
\begin{align*}
\E(\mathcal{X}_{L,U}(u)|U,t)=\sum_{n\in\Z^d} \int_{\frac1L U(Q_0+t+n)} u(x) \,dx= \int_{\R^d} u(x)\,dx,
\end{align*}

and therefore, the second term in the sum is equal to zero again. With respect to the first term, since the random variables $X_n$ are independent, we get that

\begin{align*}
\Var(\mathcal{X}_{L,U}(u)|U,t)&= \Var\big(\sum_{n\in\Z^d} u(U(X_n+n+t)/L) \big)\\
&=  \sum_{n\in\Z^d} \Var\big(u(U(X_n+n+t)/L) \big)\\
&= \left(\sum_{n\in\Z^d} \fint_{\frac1L U(Q_0+t+n)}\left|u(x)-\fint_{\frac1L U(Q_0+t+n)} u(y) \,dy\right|^2 \,dx\right)
\end{align*}

In consequence

\begin{align*}
&\Var(\mathcal{X}_{L,U}(u)|U)\\
&= \int_{\mathrm{O}(d)} \left(\sum_{n\in\Z^d} \fint_{\frac1L U(Q_0+t+n)}\left|u(x)-\fint_{\frac1L U(Q_0+t+n)} u(y) \,dy\right|^2 \,dx\right) \,dU.
\end{align*}

In conclusion, putting all together and taking $\varepsilon=\frac1L$ we obtain that

\begin{align*}
&\tau_{\frac1L}(u)\\
&= \frac{1}{L^{d-1}} \int_{Q_0}\int_{\mathrm{O}(d)} \left(\sum_{n\in\Z^d} \fint_{\frac{ U(Q_0+t+n)}L}\left|u(x)-\fint_{\frac{U(Q_0+t+n)}L} u(y)  \,dy\right|^2 \,dx\right) \,dU\, dt\\
&=\frac{1}{L^{d-1}}  \E_U\E_t\big(\Var(\mathcal{X}_{L,U}(u)|U,t))=\frac{1}{L^{d-1}} \Var(\mathcal{X}_{L,U}(u))\\
&=L \int_{\R^d}|\widehat{u}(\omega)|^2 \mathcalboondox{s}(\omega/L)\, d\omega.
\end{align*}

Finally, since $\mathcalboondox{s}(\omega)=\frac{\pi^2}{3}|\omega|^2+o(|\omega|^2)$ for $\omega\to 0$, we conclude by Theorem~\ref{Jump con Fourier} that

\[
\lim_{L\to\infty} \tau_{\frac1L}(u)=\lim_{L\to\infty}L \int_{\R^d}|\widehat{u}(\omega)|^2 \mathcalboondox{s}(\omega/L)\, d\omega=C \mathcal{J}(u),
\]
where
\begin{align*}
C=\frac{1}{2\pi^2}\int_0^\infty \frac{\mathcalboondox{s}(t)}{t^2}\, dt.
\end{align*}

For the value of the constant
\[
\mathcalboondox{s}(t)=\int_{\mathrm{O}(d)}\bigl(1-|\widehat{\phi}(tUe_1)|^2\bigr)\,dU
=\frac{1}{\omega_{d-1}}\int_{\mathbb S^{d-1}}\bigl(1-|\widehat{\phi}(tu)|^2\bigr)\,d\sigma_{d-1}(u)
\]
with $e_1$ the first standard basis vector and $\omega_{d-1}=2\pi^{d/2}/\Gamma(d/2)$ the surface area
of $\mathbb S^{d-1}.$ As $|\widehat{\phi}(x)|^2=\prod_{j=1}^d \sinc^2(x_j)=
\widehat{\psi}(x)$ for  $\psi(x)=\prod_{j=1}^d (1-|x_j|)_+$,
and as
$\psi$ is even with integral 1
and
$$\int_{\R^d}\frac{1-\cos(2\pi\,u\cdot x)}{|x|^{d+1}}\,dx =\frac{2\pi^{(d+3)/2}}{\Gamma(\frac{d+1}{2})}\,|u|,$$
we get the result.
\qed

\appendix

\section{Functions of bounded variation in \texorpdfstring{$\mathbb{R}^d$}{Rd}}\label{Appendix}

For function $u \in L^1(\R^d)$ of bounded variation its distributional derivative $Du$ is a finite vector-valued Radon measure, see \cite[Sec.3.7]{AFP00}. 

According to the Lebesgue decomposition theorem, the measure $Du$ admits a unique decomposition into two mutually singular parts: 
$$
Du = D^au + D^su,
$$
where $D^au$ is absolutely continuous with respect to the Lebesgue measure. The singular part $D^su$ is supported in $S_u$, that is,  the approximate discontinuity set of $u$. In other words,  $x\in S_u^c$ if there exists $a\in\R$ such that 
\begin{equation}
    \label{Approximate continuity}
    \lim_{\rho\to 0^+}\fint_{B_\rho(x)}|u(y)-a|\, dy=0.
\end{equation}
The singular part $D^s$ will be further decomposed into two mutually singular parts:
\begin{equation*}
    D^s = D^j u + D^c u,
\end{equation*}
The first one $D^j u$ is called the jump part, while the second one $D^c u$ is called the Cantor part. In consequence, $Du$ is divided into three mutually singular parts:
\begin{equation*}
    Du = D^a u + D^j u + D^c u.
\end{equation*}
Each of these parts is described as follows:

\subsubsection*{The absolutely continuous part $D^a u$}
It represents the "classical" gradient behavior. It is the restriction of $Du$ to the set where the function behaves regularly with respect to the Lebesgue measure $\mathcal{L}^d$, $D^a u = \nabla u \, \mathcal{L}^d$.
Here, $\nabla u$ denotes the Radon-Nikodym derivative of $Du$ with respect to $\mathcal L^d$, which belongs to $L^1(\R^d, \R^d)$. This part vanishes on any set $B \subset \R^d$ such that $\mathcal L^d(B) = 0$.

\subsubsection*{The jump part $D^j u$}
It captures the contribution of sharp discontinuities of $u$ across rectifiable hypersurfaces. Let $J_u$ denote the set of points $x$ where there is a couple of different values $u^+(x), u^-(x)$ and a unit vector $\nu(x)\in \mathbb S^{d-1}$ such that,
$u^+(x)\ne u^-(x)$ and
\begin{equation*}
    \lim_{\rho\to 0^+}\fint_{B_{\rho}^+(x)}|u(y)-u^+(x)|\, dy= \lim_{\rho\to 0^+}\fint_{B_{\rho}^-(x)}|u(y)-u^-(x)|\, dy=0,
\end{equation*}
where
\begin{equation*}
\begin{split}
    B_\rho^+(x)=\{y\in B_\rho(x):\langle y-x,\nu(x)\rangle>0\}, \\
    B_\rho^-(x)=\{y\in B_\rho(x):\langle y-x,\nu(x)\rangle<0\}.
\end{split}
\end{equation*} 
The set $J_u \subseteq S_u$ is called  the \emph{jump set} of $u$. It is a $\mathcal{H}^{d-1}$-rectifiable Borel set  such that $\mathcal{H}^{d-1}(S_u\setminus J_u)=0$.

The jump part of the distributional derivative is the measure defined as:
\begin{equation*}
    D^j u := Du \llcorner J_u = (u^+ - u^-) \otimes \nu_u \, \mathcal{H}^{d-1} \llcorner J_u
\end{equation*}
where $\mathcal{H}^{d-1}$ is the $(d-1)$-dimensional Hausdorff measure. This allows us to rewrite $Du$ as
\begin{equation*}%\label{Gradient decomposition on the jump}
    Du=\nabla u\,\mathcal{L}^d+(u^+-u^-)\nu_{J_u}\left.\mathcal{H}^{d-1}\right|_{J_u}+D^cu.
\end{equation*}

\subsubsection*{The Cantor part $D^c u$}
The Cantor part represents the singular-diffuse component of the derivative. It is the part of the singular measure $D^s u$ that does not charge the jump set. The Cantor part has the property that $D^c u(B) = 0$ for any set $B$ with $\mathcal{H}^{d-1}(B) < \infty$, yet it remains singular with respect to $\mathcal L^d$.

\medskip

As a direct consequence of the mutual singularity of the three measures $\mathcal{L}^d$, $\mathcal{H}^{d-1} \llcorner J_u$, and $D^c u$, the total variation of $u$ in $\Omega$, denoted by $|Du|(\Omega)$, can be computed as the sum of the total variations of its mutually singular components:
\begin{equation*}
    |Du|(\Omega) = \int_{\Omega} |\nabla u| \, dx + \int_{J_u \cap \Omega} |u^+ - u^-| \, d\mathcal{H}^{d-1} + |D^c u|(\Omega).
\end{equation*}

Using these definitions, we are ready to introduce the $L^2$ size of the jump

\begin{fed} Let $u\in \mathrm{BV}(\R^d) \cap L^\infty(\R^d)$, then we define the $L^2$-size of the jump of $u$ as
\begin{equation*}
    \mathcal{J}(u)=\int_{J_u}(u^{+}-u^{-})^2\, d\mathcal{H}^{d-1}.
\end{equation*}
\end{fed}

% One of the main results from \cite{BerGen2024} is the following formula for $\mathcal{J}(u)$.

% \begin{teo}\label{BG main theorem}
%      Let $u\in \mathrm{BV}(\R^d)\cap L^\infty(\R^d)$, then 
%     \begin{equation}\label{eq BG jump}
%         \lim_{R\to+\infty}\frac{1}{R}\int_{|\omega|<R}|\widehat{u}(\omega)|^2\, |\omega|^2\, d\omega = \frac{\mathcal{J}(u)}{2\pi^2}.
%     \end{equation}
% \end{teo}

As we have already mentioned in the introduction, one of the advantages of the space $\mathrm{BV}$ over the Sobolev space $W^{1,1}$ is that the former contains characteristic functions of some sets $\Omega$, which are called Caccioppoli sets. In this case, $\|\chi_\Omega\|_{\mathrm{BV}}$ is called the perimeter of $\Omega$, and we will denote it by $\Per(\Omega)$. This terminology is justified because the structure theorem of De Giorgi, see \cite{Maggi} or \cite{EvGar} for instance:
\[
\Per(\Omega)=\|\chi_\Omega\|_{\mathrm{BV}}=\H^{d-1}(\partial_*\Omega) = \mathcal{J}(\chi_\Omega),
\]
where  $\H^{d-1}$ denotes the $d-1$ dimensional Hausdorff measure and $\partial_*\Omega$ is the essential boundary (or measure-theoretic boundary) of $\Omega$ defined as $\R^d\setminus (I\cup O)$ with
\[
\begin{split}
I=\left\{x\in\R^d:\, \lim_{r\to 0} \frac{|B(x,r)\cap \Omega|}{|B(x,r)|}=1\right\},\\
O=\left\{x\in\R^d:\, \lim_{r\to 0} \frac{|B(x,r)\cap \Omega|}{|B(x,r)|}=0\right\}.
\end{split}
\]

\section*{Acknowledgments}
We would like to thank L. Gennaioli for pointing out an error in a preliminary version of this draft.

\bibliographystyle{plain}

\end{document}